\documentclass{birkjour}
\usepackage{amssymb}
\newtheorem{theorem}{Theorem}[section]
\newtheorem{lemma}[theorem]{Lemma}
\newtheorem{proposition}[theorem]{Proposition}

\newtheorem{definition}[theorem]{Definition}

\newcommand{\be}{\begin{equation}}
\newcommand{\ee}{\end{equation}}
\newcommand{\ba}{\begin{array}}
\newcommand{\ea}{\end{array}}
\newcommand{\bea}{\begin{eqnarray}}
\newcommand{\eea}{\end{eqnarray}}
\newcommand{\bee}{\begin{eqnarray*}}
\newcommand{\eee}{\end{eqnarray*}}

\newcommand{\pr}{\partial}

\newtheorem{thm}{Theorem}[section]

\def\eps{\epsilon}%
\def\tensor{\,\raise2pt\hbox{${}_{\otimes}$}\,}
\def\fdg{\,:\,}
\def\ptl{\partial}
\def\rest#1{\raise-2pt\hbox{${\lfloor_{#1}}$}}

\def\cal#1{\mathcal{#1}}

\def\ip#1#2{\langle#1,#2\rangle}

\def\grad{{\nabla}}
\newcommand{\leftexp}[2]{{\vphantom{#2}}^{#1}{#2}}

\def\halb{\frac{1}{2}}
\usepackage{color}

\renewcommand{\AA}{\mathcal A}

\begin{document}

%
%
%
%
%
%
%
%
%

\title[]
 {On Scattering for Small Data of 2+1 \\ Dimensional  Equivariant Einstein-Wave Map System}

\author[B. Dodson]{Benjamin Dodson}
\address{Department of Mathematics, Johns Hopkins University, 3400 N. Charles Street, Baltimore, MD-21218, USA}
\email{dodson@math.jhu.edu}
\thanks{(BD) is supported by NSF by grant no. DMS-1500424 and (NG) is supported by Deutsche Forschungsgemeinschaft (DFG) Postdoctoral Fellowship GU1513/1-1 to Yale University}
\author[N. Gudapati]{Nishanth Gudapati}
\address{Department of Mathematics, Yale University, 10 Hillhouse Avenue, New Haven, CT-06511, USA}
\email{nishanth.gudapati@yale.edu}






\date{January 1, 2004}

\begin{abstract}

We consider the Cauchy problem of 2+1 equivariant wave maps coupled to Einstein's equations of general relativity and prove that two separate (nonlinear) subclasses of the system disperse to their corresponding linearized equations in the large. Global asymptotic behaviour of 2+1  Einstein-wave map system is relevant because the system occurs naturally in 3+1 vacuum Einstein's equations. 

\end{abstract}
\maketitle

\section{Background and Introduction}
Let $(M, g)$ be a regular, globally hyperbolic, spatially asymptotically flat, rotationally symmetric $2+1$ dimensional Lorentzian spacetime and $(N,h)$ be a surface of revolution with the generating function $f$, then the Einstein-equivariant wave map system is defined as follows 
\begin{subequations}\label{eewm}
\begin{align} 
\mathbf{E}_{\mu \nu} =&\, \mathbf{T}_{\mu \nu} \label{eq:ee} \\
\square_{g(u)} u=&\, \frac{k^2 f_u(u) f(u)}{r^2}\label{eq:ewm},
\end{align}
\end{subequations}
where $\mathbf{E}$ is the Einstein tensor of $(M, g)$, 
\begin{align}
\mathbf{T}_{\mu \nu}  \fdg= \langle  \ptl_\mu U, \ptl_\nu U \rangle_h - \halb g_{\mu \nu} \langle
\ptl^\sigma U, \ptl_\sigma U \rangle_h
\end{align}
is the energy-momentum tensor of the equivariant wave map $U\fdg (M, g) \to (N, h) $, $U \fdg= (u, k\theta)$, $\square_g$ is the covariant wave operator defined on $(M, g)$, $r$ is the area-radius function on $(M, g)$, $f_u(u)$ is the derivative of $f(u)$ with respect to $u$. $k$ is the homotopy degree of the equivariant map which shall henceforth be assumed to be $1$. Furthermore, we assume that $f$ is a smooth, odd function such that $f(0)=0$ and $f_u(0)=1$, which for instance is admitted 
by a metric on the hyperbolic 2-plane $N = \mathbb{H}^2$.
In particular,  it follows that 
\[ f(u)f_u(u) = u + u^3 \zeta(u).\]
for some smooth function $\zeta.$


\noindent Let us assume that the metric $g$ of $M$  can be represented in the following form in null coordinate system $(\xi, \eta, \theta)$ 
\begin{align}
ds^2_g = -e^{2Z (\xi, \eta)} (d\xi\,d\eta) + r^2 (\xi,\eta) d\theta^2,
\end{align}
for some function $Z (\xi, \eta)$ and the radius function $r (\xi, \eta)$
with
 \[ r=0 \quad \text{and} \quad Z=0 \quad \text{on the axis} \quad \Gamma \quad \text{of}\,\, M.\]

\noindent Furthermore, we introduce the coordinate functions $T$ and $R$ such that 
\[ T \fdg= \halb (\xi + \eta) \quad \text{and}\quad R \fdg= \halb (\xi-\eta) \]

\noindent so that $R=0$ and $ T=\xi=\eta$ on $\Gamma$.  Further suppose that 
\[ \ptl_R r=1 \quad \text{and} \quad \ptl_R Z=0 \quad \text{on the axis} \quad \Gamma \quad \text{of}\,\, M.\]
Consequently,

\begin{align}
ds^2_g = e^{2Z (T, R)} (-dT^2 + dR^2) + r^2 (T, R) d\theta^2.
\end{align}

\noindent As calculated in \cite{diss} (cf. Section 3.5) the Einstein tensor 

\[ \mathbf{E}_{\mu \nu} = \mathbf{R}_{\mu \nu} - \halb g_{\mu \nu} R_g \] in null coordinates is given by
\begin{align*}
\mathbf{E}_{\xi\xi} =& r^{-1} (2 \ptl_\xi Z \, \ptl_\xi r- \ptl^2_\xi r),\\
\mathbf{E}_{\xi\eta} =& r^{-1} \ptl_ \xi \ptl \eta\, r,\\
\mathbf{E}_{\eta \eta} =& r^{-1} (2 \ptl_\eta Z \ptl_\eta r - \ptl^2_\eta r),\\
\mathbf{E}_{\theta \theta} =& -4 r^2 \,e^{-2 Z} \ptl_\xi \ptl_\eta Z,\\
\mathbf{E}_{\xi \theta} =& 0\,\,\text{and} \\
\mathbf{E}_{\eta\theta} =& 0.
\intertext{The components of $\mathbf{T}_{\mu \nu}$  are}
\mathbf{T}_{\xi \xi} =& \ptl_\xi u \, \ptl_\xi u,\\
\mathbf{T}_{\eta \eta} =& \ptl_\eta u \, \ptl_\eta u,  \\
\mathbf{T}_{\xi \eta} =& \frac{e^{2Z}}{4} \frac{f^2(u)}{r^2}, \\\ 
\mathbf{T}_{\theta\theta} =& \frac{r^2}{2} e^{-2Z}\left ( 4\ptl_\eta  u\ptl_\xi
u + e^{2Z}\frac{f^2(u)}{r^2} \right ). 
\end{align*} 

\noindent Furthermore, noting that 
\[ \sqrt{-g} = \halb re^{2Z (T, R)},\] 
the wave operator in null coordinates can be expressed as
\begin{align}
\square_{g} u = & \frac{1}{\sqrt{-g}} \ptl^{\nu} (\sqrt{-g} \ptl_\nu u ) \notag \\
=&- 2 e^{-2Z} r^{-1} \Big(\ptl_\eta (r \ptl_\xi u) + \ptl_\xi (r
\ptl_\eta u)\Big).
\end{align}
Therefore, from \eqref{eq:ewm} 
\begin{align}\label{eq:waveunull}
 -2 e^{-2Z} r^{-1} \Big(\ptl_\eta (r \ptl_\xi u) + \ptl_\xi (r
\ptl_\eta u)\Big) = \frac{f_u(u)f(u)}{r^2}
\end{align}
Consequently, the equivariant Einstein-wave map system can be represented as follows
\begin{subequations}\label{eewmnull}
\begin{align}
r^{-1} (2 \ptl_\xi Z \ptl_\xi r - \ptl^2_{\xi} r )=& \ptl_\xi u \, \ptl_\xi u \\
r^{-1} \ptl^2_{\xi \eta} r =&\frac{e^{2Z}}{4} \frac{f^2(u)}{r^2} 
 \label{eq:waver} \\
r^{-1} (2 \ptl_\eta Z  \ptl_\eta r - \ptl^2_\eta r)=&\ptl_\eta u \, \ptl_\eta u, \\
-4 r^2 \,e^{-2 Z} \ptl^2_{\xi \eta} Z =&\frac{r^2}{2} \left ( 4e^{-2Z} \ptl_\eta  u\ptl_\xi
u + \frac{f^2(u)}{r^2} \right ) \label{eq:waveZ}\\
\square_{g(u)} u = &\frac{f_u(u)f(u)}{r^2}. \label{eq:waveu}
\end{align}
\end{subequations}

\begin{proposition}

If we define the function $V$ such that  $R V \fdg =u$, then the following statements hold 

\begin{enumerate}
\item \begin{align}
\ptl^2_{\xi \eta} u = R \, \ptl^2_{\xi \eta} V - \halb \ptl_R V.
\end{align}
\item 
\begin{align} 
\leftexp{4+1}{\square}\, V = -4 \ptl^2_{\xi \eta} V + \frac{3}{R} (\ptl_\xi - \ptl_\eta ) V.
\end{align}
\item The wave maps equation \ref{eq:waveu} reduces to  
\begin{align}\label{waveV}
\leftexp{4+1}{\square}\, V =& \left( (e^{2Z} -1) + \left(\frac{r}{R} \ptl_\eta r + \halb \right) - \left(\frac{r}{R}\ptl_\xi r - \halb \right) \right)\frac{V}{r^2} \notag \\
&\quad + 2 \ptl_\xi V \ptl_\eta \log \left(\frac{r}{R} \right) + 2 \ptl_\eta V \ptl_\xi \log \left(\frac{r}{R} \right) \notag\\
&\quad+ e^{2Z}\frac{\zeta (u)}{r^2} R^2 V^3  
\end{align}
where $ \leftexp{4+1}{\square}$ is the wave operator on $\mathbb{R}^{4+1}$.
\end{enumerate}
\end{proposition}
\begin{proof}
The proofs of 1. and 2. immediately follow from the definitions. 
Although we shall work in the $(T, R, \theta)$ coordinates later, the proof of 3. shall be most elegant in double null coordinates.
Recall:
\begin{align}
\ptl_\xi =& \halb \left( \ptl_T + \ptl_R \right) \\
\ptl_\eta=& \halb \left(\ptl_T -\ptl_R \right) 
\end{align}  
Therefore for $u=RV$ we have, 
\begin{align}
\ptl_\xi u =&    R\ptl_\xi V + \frac{V}{2}  \\
\ptl_\eta u=&   R \ptl_\eta V -\frac{V}{2} 
\end{align}
Now consider the quantity
\begin{align}
\Big(\ptl_\eta (r \ptl_\xi u) + \ptl_\xi (r
\ptl_\eta u)\Big) =& 2r \ptl_\xi \ptl_\eta u + \ptl_\xi u \ptl_\eta r + \ptl_\eta u \ptl_\xi r \notag \\
=& 2r R  \ptl^2_{\xi \eta} V - r\ptl_R V  + R  ( \ptl_\xi V  \ptl_\eta r  + \ptl_\eta V \ptl_\xi r) \notag\\
&\quad+ \frac{V}{2}  ( \ptl_\eta r  - \ptl_\xi r ).
\end{align}
Based on our assumptions on the target manifold $(N, h)$ we have,
\begin{align}
\frac{f_u(u) f(u)}{r^2} =& \frac{u}{r^2} + \frac{\zeta(u) u^3}{r^2} 
= \frac{R}{r^2} V + \frac{\zeta(u) R^3}{r^2} V^3
\end{align}
for a smooth function $\zeta.$ Therefore, the equation \ref{eq:waveu} consecutively transforms as follows 

\begin{align*}
 -2  r^{-1} \Big(  2r R  \ptl^2_{\xi \eta} V - r\ptl_R V  + R  ( \ptl_\xi V \ptl_\eta r  + \ptl_\eta V \ptl_\xi r) + \frac{V}{2}  ( \ptl_\eta r  - \ptl_\xi r )  \Big) = e^{2Z} \frac{f_u(u)f(u)}{r^2},
\end{align*}

\begin{align}
-4 \ptl^2_{\xi \eta} V + \frac{3}{R} \ptl_R V =& \left(\frac{V}{r^2} + \frac{\zeta (u)}{r^2} R^2 V^3 \right) e^{2Z} + \frac{2}{r} (\ptl_\xi V \ptl_\eta r + \ptl_\eta V \ptl_\xi r ) \notag\\
& + \frac{V}{rR} (\ptl_\eta r - \ptl_\xi r) + \frac{1}{R} \ptl_R V \notag\\
\leftexp{4+1}{\square} V =& \left(  \frac{e^{2Z}}{r} + \frac{1}{R} (\ptl_\eta r -\ptl_\xi r)\right) \frac{V}{r} + \left(  \frac{2}{r} (\ptl_\xi V \ptl_\eta r + \ptl_\eta V \ptl_\xi r) + \frac{1}{R} \ptl_R V  \right) \notag\\
&\quad+  e^{2Z}\frac{\zeta (u)}{r^2} R^2 V^3  \notag\\
\leftexp{4+1}{\square}\, V =& \left(  \frac{e^{2Z}}{r} + \frac{1}{R} (\ptl_\eta r -\ptl_\xi r)\right) \frac{V}{r} \notag\\ 
&\quad + 2 \ptl_\xi V \ptl_\eta \log \left(\frac{r}{R} \right) + 2 \ptl_\eta V \ptl_\xi \log \left(\frac{r}{R} \right) \notag\\
&\quad+ e^{2Z}\frac{\zeta (u)}{r^2} R^2 V^3  \notag\\
\leftexp{4+1}{\square}\, V =& \left( (e^{2Z} -1) + \left(\frac{r}{R} \ptl_\eta r + \halb \right) - \left(\frac{r}{R}\ptl_\xi r - \halb \right) \right)\frac{V}{r^2} \notag \\
&\quad + 2 \ptl_\xi V \ptl_\eta \log \left(\frac{r}{R} \right) + 2 \ptl_\eta V \ptl_\xi \log \left(\frac{r}{R} \right) \notag\\
&\quad+ e^{2Z}\frac{\zeta (u)}{r^2} R^2 V^3.  
\end{align}
\end{proof}
 Thus, \eqref{waveV} is a nonlinear wave equation in the Minkowski space $\mathbb{R}^{4+1}$ which contains a critical power \footnote{In general the equation $\leftexp{n+1}{\square} \,u = u \vert u \vert^{p-1},\, u \fdg \mathbb{R}^{n+1} \to \mathbb{R}$ is critical for $p = 1+ \frac{4}{n-2}$ and $n>2$. This is because for this case the scaling symmetry of the energy matches exactly with that of the equation.}
  for a smooth function $ \zeta $ (cf. flat equivariant wave maps version \cite{jal_tah} and $(t, r, \theta)$ coordinate version \cite{diss} ).

Without loss of generality, consider the $2+1$ splitting of $M$ such that $\Sigma_0$ is the $T=0$ level set. The unit normal of $\Sigma_T$
hypersurfaces for the $\Sigma_T \hookrightarrow M $ embedding  is $\mathcal{N} \fdg =e^{-Z}\ptl_{T}$, so that $g (\mathcal{N}, \mathcal{N}) =-1$. 

In order to have well-posed initial value problem for Einstein's equations, the initial data needs to   satisfy the following constraint equations. 

\begin{subequations}\label{eq:sgwm-constraints}  
\begin{align} 
\mathbf{E}\, (\mathcal{N},\mathcal{N}) =& \mathbf{T}\, (\mathcal{N}, \mathcal{N}) \\
\mathbf{E} \, (\mathcal{N},e_i) =& \mathbf{T}\, (\mathcal{N}, e_i)
\end{align} 
\end{subequations} 
on $\Sigma_0,$ for $i=1, 2.$

Let us define the following quantities
 \begin{subequations} \label{initialdata}
 \begin{align}
 u\vert_{\Sigma_0} = u_0,&\quad \ptl_T u\vert_{\Sigma_0} =u_1, \\
 Z\vert_{\Sigma_0} =Z_0,& \quad  \ptl_T Z\vert_{\Sigma_0} =Z_1, \\
  r\vert_{\Sigma_0} =r_0, &\quad \ptl_T r \vert_{\Sigma_0} =r_1,
\end{align}
\end{subequations}
with $r_1\big \vert_{\Gamma} =0$ and $\ptl_R r_0 \big\vert_{\Gamma} =1.$
Typically, the initial constraint equations are represented in terms of the 5-tuple  $(\Sigma_0, q_0, K_0, u_0, u_1),$\footnote{where $q_0$ is the metric of $\Sigma_0$ and $K_0$ is a symmetric 2-tensor} which is directly related to \eqref{initialdata}. We are interested in the global behavior of the initial value problem of \eqref{eewm} with initial data $(\Sigma_0, q_0, K_0, u_0, u_1)$. Furthermore, we assume that the initial data is asymptotically flat as defined in \cite{AGS}. 


\noindent Define the energy as follows
\[E(t) \fdg =\int_{\Sigma_t}  \mathbf{T} (\mathcal{N}
, \mathcal{N})\, \bar{\mu}_q = \int_{\Sigma_t} \mathbf{e}\,\bar{\mu}_q \]
and let $E_0 \fdg = E (0).$ 


As a consequence of the Hardy's inequality and the aforementioned assumptions on the function $f$, the following estimates hold  
\begin{align}
 E_0 \geq& \Vert u_1 \Vert_{L^2(\mathbb{R}^2)} + \Vert u_0 \Vert_{\dot{H}^1 (\mathbb{R}^2)} \\
 \geq& \Vert V_1 \Vert_{L^2(\mathbb{R}^4)} + \Vert V_0 \Vert_{\dot{H}^1 (\mathbb{R}^4)} .
\end{align}

\noindent We are now in a position to present the Cauchy problem for the equivariant Einstein-wave map system:
\begin{equation}\label{ewmcauchy_equi}
\left. \begin{array}{rcl}
\mathbf{E}_{\mu \nu} &=&\mathbf{T}_{\mu \nu} \,\,\,\,\,\,\,\,\,\,\,\,\,\,\, \,\,\,\,\text{on}\,\, M\\
\square_g u &=& \frac{f_u(u)f(u)}{r^2} \,\,\,\,\,\,\,\, \text{on}\,\, M\\\end{array} 
\right\}
\end{equation}
with the regular, compactly supported equivariant initial data set 
\[ (\Sigma_0 ,q_0, K_0, u_0, u_1) \]
 satisfying the  constraint equations on $\Sigma_0.$ Immediately, we have the following theorem.

\begin{theorem} \label{ycb-geroch}
Let $(\Sigma_0 ,q_0, K, u_0, u_1)$ be smooth, compactly supported, equivariant initial data satisfying the constraint equations \eqref{eq:sgwm-constraints}, then there exists a regular, equivariant, globally hyperbolic maximal future development $(M, g, u)$ satisfying \eqref{ewmcauchy_equi}.
\end{theorem}
Theorem \ref{ycb-geroch} is a classic result of Choquet-Bruhat and Geroch \cite{Bruhat_Geroch_classic} . This beautiful, seminal theorem in mathematical general relativity allows us to speak about the future of the initial data but does not shed light on the global structure of $(M, g, u).$ 
As a consequence of the final result of \cite{AGS} (cf. Theorem 1.8) the following statement holds

\begin{theorem}[Global regularity of equivariant Einstein-wave maps]\label{thm:main-first}
Let $E_0 < \eps^2$ for $\eps $ sufficiently small and let
$(M, g, u)$ be the maximal Cauchy development of an asymptotically flat, compactly supported, regular Cauchy data set for the $2+1$ equivariant Einstein-wave map problem \eqref{ewmcauchy_equi} with target $(N, h)$ satisfying
\begin{equation} \label{eq:nospherecond-first} 
\int_0^s f(s') ds' \to \infty \quad \text{for} \quad s \to \infty.
\end{equation} 
Then $(M, g, u)$ is regular and causally geodesically complete.
\end{theorem} 

Actually, as a consequence of the Theorem 5.1 in \cite{AGS} (also Theorem 1.3.1 in \cite{diss}), 
Theorem 1.8  in \cite{AGS} also holds without the smallness restriction on the initial energy, with
the following additional condition on the target manifold $ (N, h)$
\begin{align}
f_s(s)f(s)s + f^2(s) >0 \quad \text{for} \quad s>0.
\end{align} 
Theorem 1.8 in \cite{AGS} carried forward the program initiated in \cite{diss} to understand
global behavior of the 2+1 wave maps coupled to Einstein's equations. The motivation to study 2+1
Einstein-wave map system comes from the fact that the system arises naturally in 3+1 vacuum Einstein's equations with one isometry group (see \cite{diss} for a detailed discussion). 
In the current work we carry the program further by addressing Open Problem 2 listed in \cite{diss} concerning the global asymptotic behavior of the 2+1 self-gravitating wave maps. 
In the general context of the initial value problem of general relativity, the question of global asymptotic behaviour is a subtle yet important question. Indeed, a comprehensive understanding of the asymptotic behaviour even in our special case shall be useful in understanding the asymptotics of more general Einstein's equations.

In precise terms, in the current work we prove that globally regular solutions of two subclasses of the system  \eqref{eewm} exhibit scattering as $T \to \infty$. These two subclasses are classified as Problem I and Problem II below. 

\section*{Problem I}

Consider a function $v$ such that
\begin{equation}\label{Problem1}
\left. \begin{array}{rcl}
\leftexp{4+1}{\square}\, v &=& F(v)\,\,\,\,\,\,\,\,\,\,\,\,\,\,\, \,\,\,\,\,\,\,\,\,\,\,\,\,\,\,\,\,\,\textnormal{on}\,\, \mathbb{R}^{4+1}\\
 v_0 = v (0, x) & \textnormal{and}& v_1 = \ptl_T v (0, x) \,\,\,\,\,\,\,\, \textnormal{on}\,\, \mathbb{R}^4\\\end{array} 
\right\}
\end{equation}
with 
\[F(v) = \left(e^{2Z} - 1 + \left(\frac{r}{R}\partial_{\eta} r + \frac{1}{2}\right) - \left(\frac{r}{R}\partial_{\xi} r - \frac{1}{2}\right)\right) \frac{v}{r^{2}} + e^{2Z} \frac{R^{2}}{r^{2}} v^{3} \zeta(R v) \]
and coupled to the equations \eqref{eq:ee} with $u = Rv.$
It may be noted that the wave equation \eqref{Problem1} is a partially linearized version\footnote{about the trivial solution $Z \equiv 0,\, r \equiv R,\, V \equiv 0$} of the fully nonlinear wave maps equation \eqref{waveV} where the linearization is applied only to the higher order terms. A special case of \eqref{Problem1} is the equation
\begin{align}
\leftexp{4+1}{\square}\, v = (e^{2Z} -1) \frac{v}{r^2} + e^{2Z} v^{3} \zeta(Rv)
\end{align} 
which corresponds to the linearization of the equation \eqref{eq:waver} (implies $r\equiv R$ with the boundary conditions on the axis $\Gamma$). Let 
\begin{align}
E(v)= \Vert v_0 \Vert_{H^1(\mathbb{R}^4)} + \Vert v_1 \Vert_{L^2(\mathbb{R}^4)},
\end{align}
we prove scattering for $v$ as follows

\begin{theorem}\label{scatteringv}
Suppose $E(v)< \eps^2$ for $\eps$ sufficiently small, then any globally regular solution $v$ of \eqref{Problem1} with 
\begin{equation}\label{conditions}
\Big|e^{2Z} - 1 \Big|, \hspace{5mm} \Big|\frac{R}{r} - 1\Big|, \hspace{5mm} \Big|R v(T, R)\Big| \leq E (v)
\end{equation}
scatters forward in time i.e., converges to a solution of its linearized equation
\begin{align}
\leftexp{4+1}{\square}\, v_{\infty} =0 
\end{align}
in the energy topology as $T \to \infty.$ 
\end{theorem}

Equivalently, scattering backwards in time can be proven using time reversal. It should be noted that the assumptions \eqref{conditions} are consistent with the results proven for the fully nonlinear system \eqref{eewm} in \cite{diss,AGS}. The proof of Theorem \ref{scatteringv} is based on an argument that the linear part of the equation \eqref{Problem1} dominates the nonlinear part in the large.  This argument in turn is based on the construction of function spaces X and Y (to be formally defined later), such that X contains a solution to the free wave equation, if $v$ lies in X then the nonlinearity lies in Y, and finally that if the nonlinearity lies in Y , then $v$ lies in X. The result then follows by the contraction mapping principle.
Our function space X exploits the endpoint Strichartz estimates, Morawetz estimates, and radial symmetry of the problem. We are able to show that if $v$ lies in this space, then the nonlinearity can be split into a term lying in $\Vert \cdot \Vert_{L^1_t L^2_x}$ and a term lying in a space that is dual to the Morawetz estimates. This implies that if the nonlinearity lies in Y, then $v$ lies in X. The details are schematically illustrated below.

\begin{thm}\label{t8_init}
If $v$ is a radial solution to the equation

\begin{equation}
\left. \begin{array}{rcl}
\leftexp{4+1}{\square}\, v &=& F(v)\,\,\,\,\,\,\,\,\,\,\,\,\,\,\, \,\,\,\,\,\,\,\,\,\,\,\,\,\,\,\,\,\,\textnormal{on}\,\, \mathbb{R}^{4+1}\\
 v_0 = v (0, x) & \textnormal{and}& v_1 = \ptl_T v (0, x) \,\,\,\,\,\,\,\, \textnormal{on}\,\, \mathbb{R}^4\\\end{array} 
\right\}
\end{equation}

\noindent then

\begin{equation}
\| v \|_{X} \leq \| v_{0} \|_{\dot{H}^{1}(\mathbb{R}^{4})} + \| v_{1} \|_{L^{2}(\mathbb{R}^{4})} + \| F \|_{Y}.
\end{equation}
\end{thm}Theorem \ref{t8_init} uses the endpoint Strichartz estimates of Keel and Tao \cite{keel_tao} and Morawetz estimates. 
\begin{lemma} [First Morawetz Estimate]
Suppose $v$ solves the linear wave equation
\begin{equation}
\left. \begin{array}{rcl}
\leftexp{4+1}{\square}\, v &=&0\,\,\,\,\,\,\,\,\,\,\,\,\,\,\, \,\,\,\,\,\,\,\,\,\,\,\,\,\,\,\,\,\,\,\,\,\,\,\,\,\,\,\textnormal{on}\,\, \mathbb{R}^{4+1}\\
 v_0 = v (0, x) & \textnormal{and}& v_1 = \ptl_T v (0, x) \,\,\,\,\,\,\,\, \textnormal{on}\,\, \mathbb{R}^4\\\end{array} 
\right\}
\end{equation}
 then 
\begin{align}
\int _{\mathbb{R}} \int_{\mathbb{R}^4} \frac{1}{\vert x \vert^3 } v^2 dx dt 
\leq  \Vert v_0 \Vert_{\dot{H}^1 (\mathbb{R})^4 } + 
\Vert v_1 \Vert_{L^2 (\mathbb{R})^4 }
\end{align}

\end{lemma}

\begin{lemma} [Second Morawetz Estimate]
Suppose $v$ solves the linear wave equation
\begin{equation}
\left. \begin{array}{rcl}
\leftexp{4+1}{\square}\, v &=&0\,\,\,\,\,\,\,\,\,\,\,\,\,\,\, \,\,\,\,\,\,\,\,\,\,\,\,\,\,\,\,\,\,\,\,\,\,\,\,\,\,\,\textnormal{on}\,\, \mathbb{R}^{4+1}\\
 v_0 = v (0, x) & \textnormal{and}& v_1 = \ptl_T v (0, x) \,\,\,\,\,\,\,\, \textnormal{on}\,\, \mathbb{R}^4\\\end{array} 
\right\}
\end{equation}
 then for a fixed $\rho>0,$ 
 
\begin{align}
&\left(\sup_{\rho} \frac{1}{\rho^{1/2}} \Big\Vert \, \nabla v \,\Big\Vert_{L_{T,x}^{2}(\mathbb{R} \times \{ |x| \leq \rho \})}\right) + \left(\sup_{\rho} \frac{1}{\rho^{1/2}} \Vert \ptl_T v \Vert_{L_{T, x}^{2}(\mathbb{R} \times \{ |x| \leq \rho \})} \right) \notag \\ 
 &\leq \| v_{0} \|_{\dot{H}^{1}(\mathbb{R}^{4})} + \| v_{1} \|_{L^{2}(\mathbb{R}^{4})}.
\end{align} 
 \end{lemma}
 We prove the Morawetz estimates using the vector fields method:
If $\check{\mathbf{T}}$ is the energy-momentum tensor of the linear wave equation for $v \fdg \mathbb{R}^{4+1} \to \mathbb{R}$, then we construct momenta or `currents' 

\[ J_{\mathfrak{X}} = \check{\mathbf{T}} (\mathfrak{X})\]
for suitable choices of Morawetz multipliers $\mathfrak{X} = \mathfrak{F} (R) \ptl_R.$ The undesirable bulk terms in the divergence of $J_\mathfrak{X}$ are corrected using the lower-order momentum 

\[J^\nu_1 [v] = \kappa   v \grad^\nu v - \halb 
v^2 \grad ^\nu \kappa\]

\noindent for suitable choices of $\kappa.$ 

Equivalent Morawetz estimates can be established for inhomogeneous and nonlinear wave equations following a similar procedure.

\noindent Formally, the function spaces $X$ and $Y$ are defined as follows

\begin{definition}[Function spaces]
Suppose $\phi(x)$ is a smooth, compactly supported, radial, decreasing function with $\phi(x) = 1$ when $|x| \leq 1$ and $\phi(x)$ is supported on $|x| \leq 2$. Then let $P_{N}$ be the Littlewood - Paley Fourier multiplier such that if $\mathcal F$ is a Fourier transform and $f$ is an $L^{1}$ function,

\begin{equation}
\mathcal F(P_{N} f)(\xi) = [\phi(\frac{\xi}{2}) - \phi(\xi)] \hat{f}(\xi),
\end{equation}

\noindent then let

\begin{align}
\| v \|_{X}^{2} \fdg =& \sum_{N} \Big\| \,P_{N} v \,\Big\|_{L_{T}^{2} L_{x}^{8}(\mathbb{R} \times \mathbb{R}^{4})}^{2} + \sum_{N} \Big\| |x|^{1/4} P_{N} v \Big\|_{L_{t}^{2} L_{x}^{16}(\mathbb{R} \times \mathbb{R}^{4})}^{2} \notag\\
&+ \sum_{N} N^{2} \Big\| P_{N} v \Big\|_{L_{T}^{\infty} L_{x}^{2}(\mathbb{R} \times \mathbb{R}^{4})}^{2} \notag\\
&+ \sum_{N} \left(\sup_{\rho > 0} \rho^{-1/2} \Big\| P_{N} \ptl_T v \Big\|_{L_{T,x}^{2}(\mathbb{R} \times \{ x : |x| \leq \rho \})}\right)^{2}  \notag \\
&+ \sum_{N} \left(\sup_{\rho > 0} \rho^{-1/2} \Big\| P_{N} \nabla_x v \Big\|_{L_{T,x}^{2}(\mathbb{R} \times \{ x : |x| \leq \rho \})}\right)^{2} \notag\\
&+ \sum_{N} N^{2} \left(\sup_{\rho > 0} \rho^{-1/2} \Big\| P_{N} v \Big\|_{L_{T,x}^{2}(\mathbb{R} \times \{ x : |x| \leq \rho \})}\right)^{2}\notag \\
&+ \sum_{N} \Big\| |x|^{-3/2} P_{N} v \Big\|_{L_{T,x}^{2}(\mathbb{R} \times \mathbb{R}^{4})}^{2} 
+ \sum_{N} N^{-2} \Big\| P_{N} \ptl_T v \Big\|_{L_{T}^{2} L_{x}^{8}(\mathbb{R} \times \mathbb{R}^{4})}^{2}.
\end{align}

\noindent Suppose $F = F_1 + F_2$
\begin{align*}
\| F \|_{Y}^{2} \fdg = & \inf_{F_{1} + F_{2} = F} \Big\| F_{1} \Big\|_{L_{T}^{1} L_{x}^{2}(\mathbb{R} \times \mathbb{R}^{4})}^{2}  \notag \\
&+ \sum_{N} \left(\sum_{j} 2^{j/2} \| P_{N} F_{2} \|_{L_{T,x}^{2}(\mathbb{R} \times \{ 2^{j} \leq |x| \leq 2^{j + 1} \})}\right)^{2}.
\end{align*}
\end{definition}
\noindent Finally, we prove the following theorem which controls the nonlinearity. The proof uses the structure of the nonlinearity in \eqref{Problem1} in a conveniently modified form using the coupled equations
\eqref{eq:ee}.
\begin{theorem}\label{t9_init}

The nonlinear wave equation

\begin{equation}
\left. \begin{array}{rcl}
\leftexp{4+1}{\square}\, v &=& F(v)\,\,\,\,\,\,\,\,\,\,\,\,\,\,\, \,\,\,\,\,\,\,\,\,\,\,\,\,\,\,\,\,\,\textnormal{on}\,\, \mathbb{R}^{4+1}\\
 v_0 = v (0, x) & \textnormal{and}& v_1 = \ptl_T v (0, x) \,\,\,\,\,\,\,\, \textnormal{on}\,\, \mathbb{R}^4\\\end{array} 
\right\}
\end{equation}

\noindent with $F(v)$  as in \eqref{Problem1}, has a solution with $\Vert v \Vert_{L^2_TL^8_x} < \infty$ for $E(v) < \eps^2$, $\eps$ sufficiently small.  
\end{theorem}

\section*{Problem II}
Consider a function $\tilde{v}$ such that 
\begin{equation} \label{Wave2}
\left. \begin{array}{rcl}
\leftexp{4+1}{\square}\, \tilde{v} &=& \tilde{F}(\tilde{v})\,\,\,\,\,\,\,\,\,\,\,\,\,\,\, \,\,\,\,\,\,\,\,\,\,\,\,\,\,\,\,\,\,\textnormal{on}\,\, \mathbb{R}^{4+1}\\
 \tilde{v}_0 = \tilde{v} (0, x) & \textnormal{and}& \tilde{v}_1 = \ptl_T \tilde{v} (0, x) \,\,\,\,\,\,\,\, \textnormal{on}\,\, \mathbb{R}^4\\\end{array} 
\right\}
\end{equation}
where 
\begin{align}
\tilde{F}(\tilde{v}) =& \left( \frac{1}{r} \ptl_\eta r + \frac{1}{2R} \right)\ptl_\xi \tilde{v}  + \left( \frac{1}{r} \ptl_\xi r - \frac{1}{2R} \right)\ptl_\eta \tilde{v} \notag\\
 &\quad + \left( \left(\frac{r}{R} \ptl_\eta r + \halb \right) - \left(\frac{r}{R} \ptl_\xi r - \halb \right) \right) \frac{\tilde{v}}{r^2} \notag\\
&\quad + \frac{R^2}{r^2} \tilde{v}^3 \zeta(R \tilde{v})
\end{align}
and $\tilde{v}$ is coupled to Einstein's equations \eqref{eewm} with $u = R \tilde{v}.$ It may be noted again that the wave equation \eqref{Wave2} is the original wave maps equation \eqref{waveV} with \eqref{eq:waveZ} linearized (implies 
$Z \equiv 0$ due to the boundary conditions on the axis $\Gamma$). Define the energy,
\begin{align}
\tilde{E}(\tilde{v})= \Vert \tilde{v}_0 \Vert_{H^1(\mathbb{R}^4)} + \Vert \tilde{v}_1 \Vert_{L^2(\mathbb{R}^4)} + \halb \Vert \tilde{v}_0\Vert_{L^4(\mathbb{R}^4)},
\end{align}
We prove scattering for \eqref{Wave2} as follows
\begin{theorem}\label{scatteringv2}
Suppose $\tilde{E}(\tilde{v})< \eps^2$ for $\eps$ sufficiently small, then any globally regular solution $\tilde{v}$ of \eqref{Problem1} with 
\begin{equation}\label{conditions2}
\Big|\frac{R}{r} - 1\Big|, \hspace{5mm}  \Big|R \tilde{v}(T, R)\Big| \leq \tilde{E} (\tilde{v})
\end{equation}
scatters forward in time i.e., converges to a solution of its linearized equation
\begin{align}
\leftexp{4+1}{\square}\, \tilde{v}_{\infty} =0 
\end{align}
in the energy topology as $T \to \infty.$ 
\end{theorem}

\noindent The proof is based on the following (nonlinear) Morawetz estimate for small data
\begin{lemma}
For any globally regular solution $\tilde{v}$ of \eqref{Wave2} 
\begin{align}
\int_{\mathbb{R}^{4+1}}  \frac{\tilde{v}^2}{\vert x \vert^3} \bar{\mu}_{\check{g}} \leq \tilde{E} (\tilde{v}).
\end{align}
\end{lemma}

\noindent This estimate directly implies

\begin{equation}
\int \int |\partial^2_{\xi \eta} r| d\xi d\eta < \infty,
\end{equation}

\noindent which then implies

\begin{equation}
\int \sup_{\xi} \Big \vert\partial_{\eta} r - \frac{1}{2} \Big\vert d\eta < \infty, \hspace{5mm} \text{and} \hspace{5mm} \int \sup_{\eta} \Big\vert\partial_{\xi} r + \frac{1}{2} \Big\vert d\xi < \infty.
\end{equation}

\noindent This fact implies that the contribution of the nonlinearity $\tilde{F}(\tilde{v})$ at large times is quite small in the energy norm, implying scattering.\vspace{5mm}

We would like to remark that the wave map field $u$ is the crucial field in the system \eqref{eewm} that drives all important geometric aspects of the evolution of the system and also the corresponding 3+1 Einstein's equations. For instance, the field $u$ was the central object of study in both non-concentration and small data arguments. Furthermore, in principle the wave map field $u$ also represents the 
nonlinear asymptotic effects of the system (e.g. nonlinear memory effect). In this regard,  Theorem \ref{t9_init}  implies that the `soul' of the system is asymptotically linear in linear approximation of either $r$ or $Z$.  

Wave maps are natural geometric generalizations of harmonic maps on one hand and linear wave equations on the other, and have been popular in the analysis and PDE community due to the nice structure and the applications in several models in mathematical physics. Thus, there exist several deep and
diverse results in the literature, focusing mainly on $\mathbb{R}^{n+1}$. In the following we discuss
 a few of these results, we refer the reader to \cite{struwe_wmsurvey, shatah_struwe, Tao_book} for instance, for detailed surveys on the study of wave maps.

 Christodoulou, Tahvildar-Zadeh and Shatah published a pioneering series of works in early 90s on equivariant and spherically symmetric wave maps on $\mathbb{R}^{2+1}$ in which they proved global existence and asymptotics for these wave maps \cite{chris_tah1, chris_tah2, jal_tah, jal_tah1}. Subsequently, it was observed in \cite{G2spacetimes} that spherically symmetric wave maps $U \fdg \mathbb{R}^{2+1} \to \mathbb{H}^2$ can be correlated to $G_2-$symmetric 3+1 dimensional spacetimes, which eventually led to a proof of strong cosmic censorship for these spacetimes.  In this context, we would like to emphasize that the nonzero homotopy degree in our case prevents us from reducing our system to flat space wave maps like in \cite{G2spacetimes}. Thus, we are forced to deal with the coupling with Einstein's equations. A detailed discussion of the occurance of 2+1 wave maps  in 3+1 spacetimes in general relativity and further sub-cases can be found in \cite{diss}.

Global existence for general wave maps was studied by Tao through a series of works \cite{Tao}. Global existence for wave maps $U\fdg \mathbb{R}^{2+1} \to \mathbb{H}^2$ for small data was proved in \cite{krieg_wmcrit}. Global existence and scattering for semilinear wave equations with power nonlinearity was proved in the classic paper of Kennig and Merle \cite{kennig_merle2008}. Global existence and scattering for wave maps $\mathbb{R}^{n+1} \to M, n= 2, 3$ was proved in  \cite{tat_besovl}. Concentration compactness for these wave maps was established in \cite{krieg_schlag_ccwm}. Likewise, large data wave maps for more general targets were studied 
in \cite{sterb_tata_long, sterb_tata_main}.

\subsection*{Notation} We shall use the Einstein's summation convention throughout. Inconsequential constants in the estimates are scaled to $1$ to avoid cluttering up the notation. For a scalar function like $v$, we shall use the notation $\ptl_T v$ and $v_T$ equivalently for partial derivatives. 

\section{Scattering for Problem I}
\subsection{Morawetz Estimates}

 Firstly, let us start with the following linear wave equation

\begin{equation}\label{4+1linwave1}
\left. \begin{array}{rcl}
\leftexp{4+1}{\square}\, v &=&0\,\,\,\,\,\,\,\,\,\,\,\,\,\,\, \,\,\,\,\,\,\,\,\,\,\,\,\,\,\,\,\,\,\,\,\,\,\,\,\,\,\,\text{on}\,\, \mathbb{R}^{4+1}\\
 v_0 = v (0, x) & \text{and}& v_1 = \ptl_T v (0, x) \,\,\,\,\,\,\,\, \text{on}\,\, \mathbb{R}^4\\\end{array} 
\right\}
\end{equation}

\noindent Denote by $\check{\mathbf{T}}$ the energy momentum tensor of $v$
\begin{align}
\check{\mathbf{T}}_{\mu \nu} (v) \fdg = \grad_\mu v \grad_\nu v - \halb \check{g}_{\mu \nu} \grad^\sigma v \grad_\sigma v,
\end{align}
where $\check{g}$ is the metric on the Minkowski space $\mathbb{R}^{4+1}.$
If we define $\check{\cal{L}} \fdg = \halb \grad^\sigma v \grad_\sigma v,$
\begin{align}
\check{\mathbf{T}}_{\mu \nu} (v) = \grad_\mu v \grad_\nu v -  \check{g}_{\mu \nu} \check{\cal{L}}.
\end{align}
We shall prove the desired Morawetz estimates for \eqref{4+1linwave1} using the vector fields method. Recall that the vector fields method is based on finding suitable spacetime multiplier vectors 
$\mathfrak{X}$ such that the corresponding momentum or `current'
\[ J^\nu_{\mathfrak{X}} = \check{\mathbf{T}}^\nu_\mu \mathfrak{X}^\mu \]
has desirable properties in view of the divergence theorem. The divergence of $J_{\mathfrak{X}}$
is given by
\begin{align}
\grad_\nu J^\nu_{\mathfrak{X}} = \halb \check{\mathbf{T}}^{\mu \nu} \,\leftexp{(\mathfrak{X})} {\pi}_{\mu \nu},
\end{align}
where we have used the fact that the energy-momentum tensor is divergence free $\grad_\nu \check{\mathbf{T}}^\nu_\mu =0$ which is a consequence of the equation \eqref{4+1linwave1}. The
tensor $\leftexp{(\mathfrak{X})} {\pi}_{\mu \nu}$ is called the deformation tensor, formally defined as
\[ \leftexp{(\mathfrak{X})} {\pi}_{\mu \nu} \fdg = L_{\mathfrak{X}} \, \check{g}_{\mu \nu} \]
where $L_{\mathfrak{X}}$ is the Lie derivative in the direction of $\mathfrak{X}.$
For the sake of brevity, let us further define $\check{e} \fdg = \check{\mathbf{T}} (\ptl_T, \ptl_T)$
and $\check{m} \fdg =  \check{\mathbf{T}} (\ptl_T, \ptl_R).$

\noindent Firstly note the multiplier $\mathfrak{X} = \ptl_T$ has the current 

\begin{align}
J_{\ptl_T} = -\check{e}\, \ptl_T + \check{m}\, \ptl_R,
\end{align}
which is divergence-free in view of the fact that $\ptl_T$ is a Killing vector of $\check{g}$, so the
`deformation' is zero
\[ \leftexp{(\ptl_T)} {\pi}_{\mu \nu} =0 \]
i.e.,
\begin{align}
\grad_\nu J^\nu_T = \halb \leftexp{(\ptl_T)}\pi_{\mu \nu} \check{\mathbf{T}}^{\mu \nu}
=0.
\end{align}
If we use this fact on the domain enclosed by two Cauchy surfaces $\check{\Sigma}_\tau$ and 
$\check{\Sigma}_s$, $s>\tau$ we have from the divergence theorem
\begin{align}
0=\int \grad_\nu J^\nu_{\mathfrak{X}} = \int_{\check{\Sigma}_\tau} \ip{\ptl_T}  { J_{\ptl_T}} \bar{\mu}_{\check{q}} -\int_{\check{\Sigma}_s} \ip{\ptl_T}  { J_{\ptl_T}} \bar{\mu}_{\check{q}}.
\end{align}
Thus we have deduced the conservation law formally, if we impose $s=T$ and $\tau =0$
 
\begin{align}
\Vert v (T) \Vert_{\dot{H}^1 (\mathbb{R}^4)} + \Vert \ptl_T v (T) \Vert_{L^2 (\mathbb{R}^4)}  = \Vert v_0 \Vert_{\dot{H}^1 (\mathbb{R}^4 )} + 
\Vert v_1 \Vert_{L^2 (\mathbb{R}^4) }.
\end{align}

 Now consider a Morawetz multiplier vector $\mathfrak{X} \fdg = \mathfrak{F}(R) \ptl_R$ so that the corresponding momentum
is given by 

\begin{align}
J_\mathfrak{X} =&\, \check{\mathbf{T}}(\mathfrak{X}) \notag \\
 =&\, \mathfrak{F}(R) \big( -\check{m} \ptl_T + \check{e} \ptl_R \big)
\end{align}

\noindent and its divergence

\begin{align}\label{mora1_div}
\grad_\nu J^\nu_\mathfrak{X} = \halb \check{\mathbf{T}}^{\mu \nu }\,\leftexp{(\mathfrak{X})}{\pi_{\mu \nu}},
\end{align}
where the non-zero terms of deformation tensor $ \pi$ are given by
\begin{align*}
\leftexp{(\mathfrak{X})}{\pi_{RR}} = 2 g_{RR} \ptl_R \mathfrak{F}(R),&  \leftexp{(\mathfrak{X})}{\pi_{\theta \theta}} =  \frac{2}{R} g_{\theta \theta} \mathfrak{F} (R), \notag \\  
\leftexp{(\mathfrak{X})}{\pi_{\phi\phi} }=  \frac{2}{R} g_{\phi \phi} \mathfrak{F}(R),& \leftexp{(\mathfrak{X})}{\pi_{\psi \psi}} =  \frac{2}{R} g_{\psi \psi} \mathfrak{F}(R).
\end{align*}
Consequently a calculation shows that  \eqref{mora1_div} can be represented as
\begin{align} \label{moragen}
\grad_\nu J^\nu_\mathfrak{X} = \left(  - \frac{6\mathfrak{F}(R)}{R}\check{\cal{L}} + \check{e}\, \ptl_R \mathfrak{F} (R) \right) 
\end{align}
Now define the following lower-order momentum vector
\begin{align}
J^\nu_1 [v] \fdg = \kappa   v \grad^\nu v - \halb
v^2 \grad ^\nu \kappa .
\end{align}
Its divergence is 
\begin{align}
\grad_\nu J^\nu_1 = &\kappa v \square v + \kappa \grad^\nu v \grad_\nu v  + v \grad^\nu v \grad_\nu \kappa - (\square \kappa )\frac{v^2}{2} - v \grad^\nu \kappa  \grad_\nu v \notag \\
=&\kappa \grad^\nu v \grad_\nu v   - (\square \kappa )\frac{v^2}{2} \notag \\
=& 2 \kappa \, \check{\cal{L}} - (\square \kappa )\frac{v^2}{2}.
\end{align}
It may noted that the momentum or `current' $J_1$ has been constructed to neutralize the undesirable 
terms in the divergence formula \eqref{moragen} while the price to pay are the lower order terms
in the spacetime integrals and boundary terms which can be handled, for instance, using the Hardy's inequality. We shall precisely do this in the following.  
\begin{lemma} [First Morawetz Estimate]\label{firstmora}
Suppose $v$ solves the linear wave equation
\begin{equation}
\left. \begin{array}{rcl}
\leftexp{4+1}{\square}\, v &=&0\,\,\,\,\,\,\,\,\,\,\,\,\,\,\, \,\,\,\,\,\,\,\,\,\,\,\,\,\,\,\,\,\,\,\,\,\,\,\,\,\,\,\textnormal{on}\,\, \mathbb{R}^{4+1}\\
 v_0 = v (0, x) & \textnormal{and}& v_1 = \ptl_T v (0, x) \,\,\,\,\,\,\,\, \textnormal{on}\,\, \mathbb{R}^4\\\end{array} 
\right\}
\end{equation}
 then 
\begin{align}
\int _{\mathbb{R}} \int_{\mathbb{R}^4} \frac{1}{\vert x \vert^3 } v^2 dx dt 
\leq  \Vert v_0 \Vert_{\dot{H}^1 (\mathbb{R})^4 } + 
\Vert v_1 \Vert_{L^2 (\mathbb{R})^4 }
\end{align}

\end{lemma}
\begin{proof}
We shall prove the theorem for a radial function $v$, the proof is essentially the same
in the general case. 
Consider the choice of $\kappa =  \frac{1}{\vert x \vert}$, then 
\begin{align}
\square \kappa = \ptl^2_R \kappa + \frac{3}{R} \ptl_R \kappa  = - \frac{1}{R^3}.
\end{align}
Now consider the case when $\mathfrak{F}(R) = \frac{1}{3}$, so that $\mathfrak{X}_1 = \frac{1}{3}\ptl_R$ then
\begin{align}
 J_{\mathfrak{X}_1} =& \frac{1}{3} \left( -\check{m} \ptl_T + \check{e} \ptl_R \right) \\
\intertext{and} 
 J_1 =& -\left(\frac{1}{R} v \ptl_T v \right)\,\ptl_T + \left( \frac{1}{R} v \ptl_R v  + \frac{v^2}{R^2} \right) \ptl_R .
\end{align}
The divergences of $J_{\mathfrak{X}_1}$ and $J_1$ are given by
\begin{align}
\grad_\nu J^\nu_{\mathfrak{X}_1} =&  -\frac{2}{R} \check{\cal{L}} \\
\grad_\nu J^\nu_1=& \frac{2}{R}\check{\cal{L}} + \frac{v^2}{2R^3} 
\end{align}
 Now consider the sum vector $J^\nu_S \fdg = J^\nu_\mathfrak{X}  + J^\nu_1 $, then it follows that 
 \begin{align}
 \grad_\nu J^\nu_S =& \grad_\nu J^\nu_\mathfrak{X}  + \grad_\nu J^\nu_1 \\
 =& \halb \frac{v^2}{R^3} 
 \end{align}
 Let us now apply the Stokes' theorem between $\check{\Sigma}_0$ and $\check{\Sigma}_T$ Cauchy surfaces,
 
 \begin{align}
\halb \int \frac{v^2}{R^3}\, \bar{\mu}_{\check{g}} =\int  \grad_\nu J^\nu_S \bar{\mu}_{\check{g}} = \int_{ {\check{\Sigma}}_0}  \ip{\ptl_T}{J_S}\bar{\mu}_{\check{q}} -\int_{\check{\Sigma}_T} \ip{\ptl_T}{ J_S} \bar{\mu}_{\check{q}}
 \end{align}
  
Now let us calculate the boundary terms 
\begin{align}
\check{g}_{\mu \nu} J^\mu_S (\ptl_T)^\nu = \check{g}_{\mu \nu} J^\mu_{\mathfrak{X}} (\ptl_T)^\nu + \check{g}_{\mu \nu} J^\mu_1 (\ptl_T)^\nu  
\end{align}   
Note that

\begin{align}
\check{g}_{\mu \nu} J^\mu_{\mathfrak{X}} (\ptl_T)^\nu = \check{g}_{T T} J_{\mathfrak{X}}^T (\ptl_T)^T = \mathfrak{F}(R) \check{m} 
\intertext{and}
\check{g}_{\mu \nu} J^\mu_1 (\ptl_T)^\nu = \check{g}_{T T} J_1^T (\ptl_T)^T = -\frac{1}{\vert  x \vert} v \, \ptl_T v.
\end{align}
From the dominant energy condition and Hardy's inequality it follows that 

\begin{align}
\int_{[0,T]} \int_{\mathbb{R}^4} \frac{1}{\vert x \vert^3} \vert v \vert^2 \bar{\mu}_{\check{g}} \leq \Vert v_0 \Vert_{\dot{H}^1 (\mathbb{R}^4) } + 
\Vert v_1 \Vert_{L^2 (\mathbb{R}^4)}.
\end{align}
As the right is independent of $T$, taking the limit $T \to \infty$ and time reversal, it follows that
\begin{align}
\int_{\mathbb{R}}\int_{\mathbb{R}^4} \frac{1}{\vert x \vert^3} \vert v \vert^2 \bar{\mu}_{\check{g}} \leq \Vert v_0 \Vert_{\dot{H}^1 (\mathbb{R}^4) } + 
\Vert v_1 \Vert_{L^2 (\mathbb{R}^4) }.
\end{align}
\end{proof}

\begin{lemma} [Second Morawetz Estimate]\label{t3}
Suppose $v$ solves the linear wave equation
\begin{equation}
\left. \begin{array}{rcl}
\leftexp{4+1}{\square}\, v &=&0\,\,\,\,\,\,\,\,\,\,\,\,\,\,\, \,\,\,\,\,\,\,\,\,\,\,\,\,\,\,\,\,\,\,\,\,\,\,\,\,\,\,\textnormal{on}\,\, \mathbb{R}^{4+1}\\
 v_0 = v (0, x) & \textnormal{and}& v_1 = \ptl_T v (0, x) \,\,\,\,\,\,\,\, \textnormal{on}\,\, \mathbb{R}^4\\\end{array} 
\right\}
\end{equation}
 then for a fixed $\rho>0,$ 
 
\begin{align}\label{3.2}
&\left(\sup_{\rho} \frac{1}{\rho^{1/2}} \Big\Vert \, \nabla v \,\Big\Vert_{L_{T,x}^{2}(\mathbb{R} \times \{ |x| \leq \rho \})}\right) + \left(\sup_{\rho} \frac{1}{\rho^{1/2}} \Vert \ptl_T v \Vert_{L_{T, x}^{2}(\mathbb{R} \times \{ |x| \leq \rho \})} \right) \notag \\ 
 &\leq \| v_{0} \|_{\dot{H}^{1}(\mathbb{R}^{4})} + \| v_{1} \|_{L^{2}(\mathbb{R}^{4})}.
\end{align} 
 \end{lemma}
 \noindent \emph{Proof:}
 Let $\chi \in C_0^{\infty} (\mathbb{R}^4)$ be a positive, radially symmetric function  such
 that 
 \begin{align}
 \chi (x) =1, \quad \vert x \vert \leq 1, \quad \chi (x) = \frac{1}{\vert x \vert} \quad \text{for} \quad \vert x \vert > 2, 
 \end{align}
 so that
 \begin{align} 
 \psi(R) \fdg= \frac{d}{dR} (R \cdot\chi(R)) \geq 0, \quad \forall R \in (0, \infty).
 \end{align}
 Notice that 
 $\psi(R) \geq 0$ is supported on $R\leq 2$ and $\psi(R)=1$ for $R\leq 1.$
 
\noindent Consider the Morawetz multiplier $\mathfrak{X}_2 =  \mathfrak{F}(R) \ptl_R$ for $\mathfrak{F}(R)= \frac{1}{3} R  \cdot\chi(\frac{x}{\rho})$. Then 
 \begin{align}
\grad_\nu J^\nu_{\mathfrak{X}_2} =& \left(-2 \chi \left(\frac{x}{\rho}\right) \cal{L} + \psi \left(\frac{x}{\rho}\right) \check{e}\right)
\end{align}
 
\noindent Consider the lower order momentum 
\begin{align}
J^\nu_{\kappa_2} [v]\fdg = \kappa_2   v \grad^\nu v - 
v^2 \grad ^\nu \kappa_2  
\end{align}
with $\kappa_2 = \chi(\frac{ x }{\rho}).$
Consequently, 

\begin{align}
\grad_\nu J^\nu_{\kappa_2} =&\kappa \grad^\nu v \grad_\nu v   - (\square \kappa )\frac{v^2}{2}  \\
=& 2\chi \left(\frac{x}{\rho}\right)\cal{L}  - \Delta \chi \left(\frac{x}{\rho}\right) \frac{v^2}{2}.
\end{align}
The divergence of the sum is then 
 \begin{align}
 \grad_\nu J^\nu_{S_2} =& \grad_\nu J^\nu_{\mathfrak{X}_2}  + \grad_\nu J^\nu_{\kappa_2} \\
 =& \psi\left(\frac{x}{\rho}\right) \check{e}
  - \Delta \chi \left(\frac{x}{\rho}\right) \frac{v^2}{2}
 \end{align}
 
\noindent  Let us now use the divergence theorem between two Cauchy surfaces
  \begin{align}
\int  \grad_\nu J^\nu_{S_2} \bar{\mu}_{\check{g}} = \int_{ {\check{\Sigma}}_0}  \ip{\ptl_T}{J_{S_2}}\bar{\mu}_{\check{q}} -\int_{\check{\Sigma}_T} \ip{\ptl_T}{ J_{S_2}} \bar{\mu}_{\check{q}}
 \end{align}
 Moreover,
 \begin{align}
  \int  \grad_\nu J^\nu_{S_2} \bar{\mu}_{\check{g}} \leq &
  \halb \int \psi \left(\frac{R}{\rho}\right) (\ptl_T v)^2 \bar{\mu}_{\check{g}} + \halb \int \psi \left(\frac{R}{\rho}\right) \vert\grad_x v\vert^2 \bar{\mu}_{\check{g}} \\
&\quad  + c \int_{\mathbb{R}}\int_{\vert x \vert >\rho} \frac{\rho}{\vert x \vert^3} \vert v \vert^2 \bar{\mu}_{\check{g}}.
 \end{align}
 By previous lemma it follows that
 \begin{align}
 \int_{\mathbb{R}} \int_{\vert x \vert>\rho }  \frac{\rho}{\vert x \vert^3} v^2 \bar{\mu}_{\check{g}} \leq 
 \rho \left(  \Vert v_0 \Vert_{\dot{H}^1 (\mathbb{R})^4 } + 
\Vert v_1 \Vert_{L^2 (\mathbb{R})^4 }\right)
  \end{align}
By the Hardy's theorem and the dominant energy condition, the boundary terms can
be estimated by the initial energy. Therefore,

 \begin{align}
\int_{\mathbb{R}} \int _{\vert x \vert<\rho} (\ptl_T v)^2 + \vert \grad_x  v \vert^2 \bar{\mu}_{\check{g}}
\leq \rho \left(  \Vert v_0 \Vert_{\dot{H}^1 (\mathbb{R})^4 } + 
\Vert v_1 \Vert_{L^2 (\mathbb{R})^4 }\right).
 \end{align}
The result of the theorem now follows.

\begin{theorem}
Suppose that $v$ is a solution to the inhomogeneous wave equation
\begin{equation}
\left. \begin{array}{rcl}
\leftexp{4+1}{\square}\, v &=&F\,\,\,\,\,\,\,\,\,\,\,\,\,\,\, \,\,\,\,\,\,\,\,\,\,\,\,\,\,\,\,\,\,\,\,\,\,\,\,\,\,\,\textnormal{on}\,\, \mathbb{R}^{4+1}\\
 v_0 = v (0, x) & \textnormal{and}& v_1 = \ptl_T v (0, x) \,\,\,\,\,\,\,\, \textnormal{on}\,\, \mathbb{R}^4\\\end{array} 
\right\}
\end{equation}

\noindent Then

\begin{align}\label{5.2}
& \| v \|_{L_{t}^{\infty} \dot{H}^{1}(\mathbb{R} \times \mathbb{R}^{4})} + \| v_{t} \|_{L_{t}^{\infty} L_{x}^{2}(\mathbb{R} \times \mathbb{R}^{4})}  \notag \\
&\leq \| v_{0} \|_{\dot{H}^{1}(\mathbb{R}^{4})} + \| v_{1} \|_{L_{x}^{2}(\times \mathbb{R}^{4})}  +  \left(\sum_{j} 2^{j/2} \| F \|_{L_{t,x}^{2}(\mathbb{R} \times \{ x : 2^{j} \leq |x| \leq 2^{j + 1} \})}\right),
\end{align}
and 
\begin{align}\label{5.3}
&\Big\Vert |x|^{-3/2} v \Big\Vert_{L_{t,x}^{2}(\mathbb{R} \times \mathbb{R}^{4})} + \left(\sup_{\rho > 0} \rho^{-1/2} \| \nabla v \|_{L_{t,x}^{2}(\mathbb{R} \times \{ x : |x| \leq \rho \})}\right) \notag\\
&+ \left(\sup_{\rho > 0} \rho^{-1/2} \| v_{t} \|_{L_{T,x}^{2}(\mathbb{R} \times \{ x : |x| \leq \rho \})}\right) \notag\\
&\leq \| v_{0} \|_{\dot{H}^{1}(\mathbb{R}^{4})} + \| v_{1} \|_{L^{2}(\mathbb{R}^{4})} + \left(\sum_{j} 2^{j/2} \| F \|_{L_{T,x}^{2}(\mathbb{R} \times \{ x : 2^{j} \leq |x| \leq 2^{j + 1} \})}
\right),
\end{align}

\end{theorem}

\noindent \emph{Proof:} We start with $(\ref{5.2})$, which is the dual of $(\ref{3.2})$. If $f \in L^{2}(\mathbb{R}^{4})$ then

\begin{align}\label{5.5}
&\langle f, \nabla \int_{0}^{t} \frac{\sin((t - \tau) \sqrt{-\Delta})}{\sqrt{-\Delta}} F(\tau) d\tau \rangle 
\notag \\
&= \int_{0}^{t} \langle \nabla \frac{\sin((t - \tau) \sqrt{-\Delta})}{\sqrt{-\Delta}} f, F(\tau) \rangle d\tau \notag\\
&\leq  \left(\sup_{j} 2^{-j/2} \| \nabla \frac{\sin((t - \tau) \sqrt{-\Delta})}{\sqrt{-\Delta}} f \|_{L_{t,x}^{2}(\mathbb{R} \times \{ x : 2^{j} \leq |x| \leq 2^{j + 1} \})} \right)\notag \\
&\quad \cdot \left(\sum_{j} 2^{j/2} \| F \|_{L_{t,x}^{2}(\mathbb{R} \times \{ x : 2^{j} \leq |x| \leq 2^{j + 1} \})}\right) \notag \\
&\leq  \| f \|_{L^{2}(\mathbb{R}^{4})} \left(\sum_{j} 2^{j/2} \| F \|_{L_{T, x}^{2}(\mathbb{R} \times \{ x : 2^{j} \leq |x| \leq 2^{j + 1} \})}\right).
\end{align}

\noindent An identical computation also proves

\begin{align}\label{5.8}
&\langle f, \partial_{t} \int_{0}^{t} \frac{\sin((t - \tau) \sqrt{-\Delta})}{\sqrt{-\Delta}} F(\tau) d\tau \rangle \notag \\
=& \langle f, \int_{0}^{t} \cos((t - \tau) \sqrt{-\Delta}) F(\tau) d\tau \rangle \notag\\
= &\int_{\mathbb{R}} \langle \int_{t}^{\infty} \cos((t - \tau) \sqrt{-\Delta}) f dt, F(\tau) \rangle d\tau \notag\\ 
\leq & \| f \|_{L^{2}(\mathbb{R}^{4})} \left(\sum_{j} 2^{j/2} \| F \|_{L_{t,x}^{2}(\mathbb{R} \times \{ x : 2^{j} \leq |x| \leq 2^{j + 1} \})}\right).
\end{align}

\noindent $(\ref{5.3})$ is proved in a similar way as on the Morawetz estimates in Lemmas $\ref{firstmora}$ and $\ref{t3}$
but adjusting the identities and estimates for $\square v = F.$

\begin{align}\label{5.10}
 \int 
 \frac{1}{\vert x\vert^3} |v(t,x)|^{2} dx dt \leq& \, \| v_{0} \|_{\dot{H}^{1}(\mathbb{R}^{4})}^{2} +
  \| v_{1} \|_{L_{x}^{2}(\mathbb{R}^{4})}^{2} +\int |\nabla v(t,x)| |F(t,x)| dxdt \notag \\
  &\quad + \int \frac{1}{|x|} |v(t,x)| |F(t,x)| dxdt
\end{align}

\noindent  by  $(\ref{5.2})$,

\begin{align}\label{5.11}
\int \int \frac{1}{|x|^{3}} |v(t,x)|^{2} dx dt \leq& \| v_{0} \|_{\dot{H}^{1}(\mathbb{R}^{4})}^{2} + \| v_{1} \|_{L_{x}^{2}(\mathbb{R}^{4})}^{2}  \notag\\
&+ \left(\int \int \frac{1}{|x|^{3}} |v(t,x)|^{2} dx dt\right)^{1/2} \\
&\quad\quad\quad\cdot \left(\sum_{j} 2^{j/2} \| F \|_{L_{t,x}^{2}(\mathbb{R} \times \{ x : 2^{j} \leq |x| \leq 2^{j + 1} \})}\right) \notag\\ 
&+ \left(\sup_{j} 2^{-j/2} \| \nabla v \|_{L_{t,x}^{2}(\mathbb{R} \times \{ x : 2^{j} \leq |x| \leq 2^{j + 1} \})}\right) \notag\\
&\quad\quad\quad\cdot\left(\sum_{j} 2^{j/2} \| F \|_{L_{t,x}^{2}(\mathbb{R} \times \{ x : 2^{j} \leq |x| \leq 2^{j + 1} \})}\right).
\end{align}

\noindent As in Lemma $\ref{t3}$, again this time using $\square v = F$,

\begin{align}\label{5.13}
&\frac{1}{\rho} \int \chi(\frac{x}{\rho}) [v_{t}(t,x)^{2} + |\nabla v(t,x)|^{2}] dxdt  + C \int \frac{1}{|x|^{3}}
|v(t,x)|^{2} dxdt\notag \\
&\quad \leq \| v_{0} \|_{\dot{H}^{1}(\mathbb{R}^{4})}^{2} + \| v_{1} \|_{L_{x}^{2}(\mathbb{R}^{4})}^{2} \notag \\ & \quad \quad + \int |\nabla v(t,x)| |F(t,x)| dx + \int \frac{1}{|x|} |v(t,x)| |F(t,x)| dx.
\end{align}
where $\chi(x)$is an in Lemma 2.2.
Consequently,
\begin{align}\label{5.14}
&\frac{1}{\rho} \int_{\mathbb{R}} \int_{|x| \leq \rho} [|\nabla v(t,x)|^{2} + |v_{t}(t,x)|^{2}] dx dt \notag
\notag \\ 
\leq& \| v_{0} \|_{\dot{H}^{1}(\mathbb{R}^{4})}^{2} + \| v_{1} \|_{L_{x}^{2}(\mathbb{R}^{4})}^{2} + \int_{\mathbb{R}} \int \frac{1}{|x|^{3}} |v(t,x)|^{2} dx dt \notag\\ 
&+ \left(\sum_{j} 2^{j/2} \| F \|_{L_{t,x}^{2}(\mathbb{R} \times \{ x : 2^{j} \leq |x| \leq 2^{j + 1} \})}\right)^{2}  \notag \\ 
&+ \left( \sup_{j} 2^{-j/2} \| \nabla v \|_{L_{t,x}^{2}(\mathbb{R} \times \{ x : 2^{j} \leq |x| \leq 2^{j + 1} \})} \right) \cdot \notag\\ 
&\quad\quad\quad\cdot \left(\sum_{j} 2^{j/2} \| F \|_{L_{t,x}^{2}(\mathbb{R} \times \{ x : 2^{j} \leq |x| \leq 2^{j + 1} \})}\right) \notag \\
&+ \left(\int_{\mathbb{R}} \int \frac{1}{|x|^{3}} |u(t,x)|^{2} dx dt \right)^{1/2} \cdot \notag\\ 
&\quad\quad\quad \cdot \left(\sum_{j} 2^{j/2} \| F \|_{L_{t,x}^{2}(\mathbb{R} \times \{ x : 2^{j} \leq |x| \leq 2^{j + 1} \})}\right).
\end{align}

\noindent Combining $(\ref{5.11})$ and $(\ref{5.14})$ proves $(\ref{5.3})$.\vspace{5mm}

\subsection{Strichartz Esimates}

\begin{theorem}[Endpoint Strichartz estimate]\label{t1}
Suppose that $v$ solves the inhomogeneous wave equation

\begin{equation}
\left. \begin{array}{rcl}
\leftexp{4+1}{\square}\, v &=&F_1 + F_2\,\,\,\,\,\,\,\,\,\,\,\,\,\,\, \,\,\,\,\,\,\,\,\,\,\,\,\,\,\,\,\,\,\,\,\,\,\,\,\,\,\,\textnormal{on}\,\, \mathbb{R}^{4+1}\\
 v_0 = v (0, x) \in \dot{H}^{1}(\mathbb{R}^{4})  & \textnormal{and}& v_1 = \ptl_T v (0, x)
 \in L^{2}(\mathbb{R}^{4}) \\\end{array} 
\right\}
\end{equation}


\noindent Then,

\begin{align}\label{1.2}
\| v \|_{L_{T}^{2} L_{x}^{8}(\mathbb{R} \times \mathbb{R}^{4})} \leq  \| v_{0} \|_{\dot{H}^{1}(\mathbb{R}^{4})} + \| v_{1} \|_{L^{2}(\mathbb{R}^{4})} +& \| \nabla F_{1} \|_{L_{T}^{2} L_{x}^{8/7}(\mathbb{R} \times \mathbb{R}^{4})} \notag\\  +& \| F_{2} \|_{L_{T}^{1} L_{x}^{2}(\mathbb{R} \times \mathbb{R}^{4})}.
\end{align}
\end{theorem}

\noindent \emph{Proof:} See Keel and Tao \cite{keel_tao}. $\Box$
\begin{theorem}[Radially symmetric Strichartz estimate]\label{t4}
Suppose $v$ solves the wave equation

\begin{equation}\label{4.1}
\left. \begin{array}{rcl}
\leftexp{4+1}{\square}\, v &=& 0\,\,\,\,\,\,\,\,\,\,\,\,\,\,\, \,\,\,\,\,\,\,\,\,\,\,\,\,\,\,\,\,\,\,\,\,\,\,\,\,\,\,\textnormal{on}\,\, \mathbb{R}^{4+1}\\
 v_0 = v (0, x) & \textnormal{and}& v_1 = \ptl_T v (0, x) \,\,\,\,\,\,\,\, \textnormal{on}\,\, \mathbb{R}^4\\\end{array} 
\right\}
\end{equation}

\noindent with $v_{0}$ and $v_{1}$ radially symmetric. Then

\begin{equation}\label{4.2}
\Big\| |x|^{1/2} v \Big\|_{L_{T}^{2} L_{x}^{\infty}(\mathbb{R} \times \mathbb{R}^{4})} \leq\| v_{0} \|_{\dot{H}^{1}(\mathbb{R}^{4})} + \| v_{1} \|_{L^{2}(\mathbb{R}^{4})}.
\end{equation}
\end{theorem}

\noindent \emph{Proof:} To prove this for $|x| >> T$ we only need to use Hardy's inequality, finite propagation speed, and the Sobolev embedding theorem. Suppose $|x| \geq 32T$ and make a partition of unity.\vspace{5mm}

\noindent Let $\phi \in C_{0}^{\infty}(\mathbb{R}^{4})$ be a radial, decreasing function, with $\phi(x) = 1$ for $|x| \leq 1$ and $\phi(x)$ is supported on $|x| \leq 2$. Then let

\begin{equation}\label{4.31}
\chi(x) = \phi(\frac{x}{2}) - \phi(x).
\end{equation}

\noindent If $x \neq 0$,

\begin{equation}\label{4.32}
1 = \sum_{j \in \mathbb{Z}} \chi(2^{j} x),
\end{equation}

\noindent and for each $k \in \mathbb{Z}$ let

\begin{equation}
\aligned
\chi_{k}(x) = \chi(2^{-k} x), \\
\tilde{\chi}_{k}(x) = \chi(2^{-k + 2} x) + \chi(2^{-k + 1} x) + \chi(2^{-k} x) + \chi(2^{-k - 1} x) + \chi(2^{-k - 2} x).
\endaligned
\end{equation}



\noindent Then by finite propagation speed, for $0 \leq T \leq 2^{k - 4}$ and $2^{k} \leq |x| \leq 2^{k + 1}$,

\begin{equation}\label{4.33}
\chi_{k}(x) v(T,x) = \cos(T \sqrt{-\Delta}) \tilde{\chi}_{k}(x) v_0(x) + \frac{\sin(T \sqrt{-\Delta})}{\sqrt{-\Delta}} \tilde{\chi}_{k}(x) v_1(x).
\end{equation}

\noindent Then by the radial Sobolev embedding theorem and conservation of energy,

\begin{align}\label{4.34}
&\Big \| |x| \chi_{k}(x) v(T,x) \Big\|_{L_{T}^{\infty} L_{x}^{\infty}([0, 2^{k - 4}] \times \{ x : 2^{k} \leq |x| \leq 2^{k + 1} \})}  \\
\leq& \Big\| \tilde{\chi}_{k} v_0 \Big\|_{\dot{H}^{1}(\mathbb{R}^{4})} + \Big\| \tilde{\chi}_{k} v_1 \Big\|_{L^{2}(\mathbb{R}^{4})} \\
\leq& \Big\| \nabla v_0 \Big\|_{L^{2}(2^{k - 2} \leq |x| \leq 2^{k + 4})} + \Bigg\| \frac{1}{|x|} v_0 \Bigg\|_{L^{2}(2^{k - 2} \leq |x| \leq 2^{k + 4})}  \notag \\ & \quad + \Big\|\, v_1\, \Big\|_{L^{2}(2^{k - 2} \leq |x| \leq 2^{k + 4})},
\end{align}

\noindent and by H{\"o}lder's inequality

\begin{align}\label{4.35}
&\Big\| |x|^{1/2} v(T,x) \Big\|_{L_{T}^{2} L_{x}^{\infty}([0, 2^{k - 4}] \times \{ x : 2^{k} \leq |x| \leq 2^{k + 1} \})} \notag\\
& \leq \Big\| \nabla v_0 \Big \|_{L^{2}(2^{k - 2} \leq |x| \leq 2^{k + 4})}  + \Bigg\| \frac{1}{|x|} v_0 \Bigg\|_{L^{2}(2^{k - 2} \leq |x| \leq 2^{k + 4})} \notag \\ & \quad+ \Big\|\, v_1 \,\Big\|_{L^{2}(2^{k - 2} \leq |x| \leq 2^{k + 4})}.
\end{align}

\noindent Then by Hardy's inequality and $(\ref{4.33})$,

\begin{align}\label{4.36}
&\Big\| |x|^{1/2} v(T,x) \Big\|_{L_{T}^{2} L_{x}^{\infty}(\mathbb{R} \times \{ x : |x| \geq 32 T\})}^{2} \notag \leq \sum_{k} 2^{k} \| |x|^{1/2} \chi_{k}(x) v(T,x) \|_{L_{T,x}^{\infty}}^{2} \\
&\leq \sum_{k} \Big\| \nabla v_0 \Big\|_{L^{2}(2^{k - 2} \leq |x| \leq 2^{k + 4})}^{2}  + \Bigg\| \frac{1}{|x|} v_0 \Bigg\|_{L^{2}(2^{k - 2} \leq |x| \leq 2^{k + 4})}^{2}  \notag \\ & \quad \quad +\Big \| \,v_1\, \Big\|_{L^{2}(2^{k - 2} \leq |x| \leq 2^{k + 4})}^{2} \notag\\
&\leq \| \nabla v_0 \|_{L^{2}(\mathbb{R}^{4})}^{2} + \| v_1 \|_{L^{2}(\mathbb{R}^{4})}^{2}.
\end{align}

\noindent \textbf{Remark:} This estimate is not necessarily sharp in this region.\vspace{5mm}

\noindent Now consider $\vert x \vert \leq 32T$  and suppose $v_{0} = 0$ and $v_{1} = g \in L^{2}(\mathbb{R}^{4})$. Without loss of generality suppose $T > 0$. Then by the fundamental solution to the wave equation (see for example Sogge \cite{sogge_book}),

\begin{equation}\label{4.3}
u(T,x) = 3T \int_{|y| < 1} \frac{g(x + Ty)}{(1 - |y|^{2})^{1/2}} dy + 
T^{2} \int_{|y| < 1} \frac{y \cdot (\nabla g)(x + Ty)}{(1 - |y|^{2})^{1/2}} dy.
\end{equation}

\noindent Suppose $\omega$ is the surface area of the unit sphere $S^{3} \subset \mathbb{R}^{4}$. If $g$ is radial then $v$ is radial, so

\begin{equation}\label{4.4}
v(T,x) = \frac{1}{\omega |x|^{3}} \int_{\partial B(0, |x|)} v(T,z) d\sigma(z).
\end{equation}

\noindent $g$ radial implies that $g(y_{1}, y_{2}, y_{3}, y_{4}) = g(|y_{1}|, (y_{2}^{2} + y_{3}^{2} + y_{4}^{2})^{1/2})$. For any $y \in \mathbb{R}^{4}$, $|y| < 1$ let $\cal{R}$ be the rotation matrix that rotates $y \in \mathbb{R}^{4}$ to the vector $|y| e_{1}$, where $e_{1} = (1, 0, 0, 0)$. $g$ radial implies

\begin{equation}\label{4.5}
\aligned
\frac{3T}{\omega |x|^{3}} \int_{|y| < 1} \frac{1}{(1 - |y|^{2})^{1/2}} \int_{\partial B(0, |x|)} g(x + 
Ty) d\sigma(x) dy \\ = \frac{3T}{\omega |x|^{3}} \int_{|y| < 1} \frac{1}{(1 - |y|^{2})^{1/2}} \int_{\partial B(0, |x|)} g(\cal{R}(x + Ty)) d\sigma(x) dy \\ = 3 T \int_{0}^{1} \frac{y_{1}^{3}}{(1 - y_{1}^{2})^{1/2}} \int_{-\frac{\pi}{2}}^{\frac{\pi}{2}} g(T y_{1} + |x| \sin(\theta), |x| \cos(\theta)) \cos(\theta)^{2} d\theta dy_{1}.
\endaligned
\end{equation}

\noindent Making a change of variables,

\begin{equation}\label{4.6}
= 3T \int_{0}^{1} \frac{y_{1}^{3}}{(1 - y_{1}^{2})^{1/2}} \int_{-1}^{1} g(T y_{1} + |x| v, |x| (1 - v^{2})^{1/2}) (1 - v^{2})^{1/2} dv dy_{1}.
\end{equation}

\noindent Choosing 

\begin{equation}\label{variable}
\aligned
s^{2} = (Ty_{1} + |x| v)^{2} + |x|^{2} (1 - v^{2}) &= T^{2} y_{1}^{2} + |x|^{2} + 2T |x| y_{1} v, \\ 
s ds &= T |x| y_{1} dv,
\endaligned
\end{equation}

\begin{equation}\label{4.7}
(\ref{4.6}) = \frac{3}{|x|} \int_{0}^{1} \frac{y_{1}^{2}}{(1 - y_{1}^{2})^{1/2}} \int_{||x| - T y_{1}|}^{|x| + T y_{1}} g(s) s (1 - v(s)^{2})^{1/2} ds dy_{1}.
\end{equation}

\noindent Then since $|x|^{1/2} (1 - v^{2})^{1/2} \leq s^{1/2}$,

\begin{equation}\label{4.8}
|x|^{1/2} (\ref{4.7}) \leq \int_{0}^{1} \frac{y_{1}^{2}}{(1 - y_{1}^{2})^{1/2}} \mathcal M(g(s) s^{3/2})(T y_{1}),
\end{equation}

\noindent where $\mathcal M$ is the Maximal function in one dimension,

\begin{equation}
\mathcal M f(x) = \sup_{R > 0} \frac{1}{R} \int_{x - R}^{x + R} f(t) dt.
\end{equation}

\noindent It is a well - known fact (see for example \cite{taylor_book}) that for any $1 < p \leq \infty$,

\begin{equation}
\| \mathcal M f \|_{L^{p}(\mathbf{R})} \lesssim_{p} \| f \|_{L^{p}(\mathbf{R})}.
\end{equation}

\noindent Then $g \in L^{2}(\mathbb{R}^{4})$ radial implies $g(s) s^{3/2} \in L^{2}([0, \infty))$, so by a change of variables $\| \mathcal M(g(s) s^{3/2})(T y_{1}) \|_{L_{T}^{2}(\mathbb{R})} \leq \frac{1}{\sqrt{y_{1}}}$, so

\begin{equation}\label{4.9}
\| (\ref{4.8}) \|_{L_{T}^{2}(\mathbb{R})} \leq \int_{0}^{1} \frac{y_{1}^{3/2}}{(1 - y_{1}^{2})^{1/2}} \| g \|_{L^{2}(\mathbb{R}^{4})} dy_{1} \leq \| g \|_{L^{2}(\mathbb{R}^{4})}.
\end{equation}

\noindent Next we compute

\begin{equation}\label{4.10}
T^{2} \int_{|y| < 1} \frac{y \cdot (\nabla g)(x + Ty)}{(1 - |y|^{2})^{1/2}} dy.
\end{equation}

\noindent It will be convenient to split this integral into two pieces,

\begin{align}\label{4.11}
&=  T^{2} \int_{\sup(\frac{5}{6}, 1 - \frac{|x|}{2T}) \leq |y| \leq 1} \frac{y \cdot (\nabla g)(x + Ty)}{(1 - |y|^{2})^{1/2}} dy  \\ &\quad + T^{2} \int_{|y| \leq \sup(\frac{5}{6}, 1 - \frac{|x|}{2T})} \frac{y \cdot (\nabla g)(x + Ty)}{(1 - |y|^{2})^{1/2}} dy.
\end{align}

\noindent Since $g$ is radially symmetric, making the change of variables $(\ref{variable})$,

\begin{align}\label{4.12}
&\frac{1}{\omega |x|^{3}} \int_{\partial B(0, |x|)} T^{2} \int_{\sup(\frac{5}{6}, 1 - \frac{|x|}{2T}) \leq |y| \leq 1} \frac{y \cdot (\nabla g)(x + Ty)}{(1 - |y|^{2})^{1/2}} dy d\sigma(x) \\ =& \frac{T}{|x|} \int_{\sup(\frac{5}{6}, 1 - \frac{|x|}{2T})}^{1} \int_{||x| - Ty_{1}|}^{|x| + Ty_{1}} \frac{y_{1}^{3}}{(1 - y_{1}^{2})^{1/2}} g'(s) (|x| v(s) + T y_{1}) (1 - v(s)^{2})^{1/2} ds dy_{1}.
\end{align}

\noindent Since $(\ref{variable})$ implies $(1 - v^{2}) = 0$ when $s = |x| + T y_{1}$ or $||x| - T y_{1}|$, integrating by parts,

\begin{equation}\label{4.13}
= \frac{-1}{|x|} \int_{\sup(\frac{5}{6}, 1 - \frac{|x|}{2T})}^{1} \int_{||x| - Ty_{1}|}^{|x| + Ty_{1}} \frac{y_{1}^{3}}{(1 - y_{1}^{2})^{1/2}} g(s) s (1 - v(s)^{2})^{1/2} ds dy_{1}
\end{equation}

\begin{equation}\label{4.14}
+ \frac{1}{|x|^{2}} \int_{\sup(\frac{5}{6}, 1 - \frac{|x|}{2T})}^{1} \int_{||x| - Ty_{1}|}^{|x| + Ty_{1}} \frac{y_{1}^{3}}{(1 - y_{1}^{2})^{1/2}} g(s) s (|x| v(s) + T y_{1}) v(s) (1 - v(s)^{2})^{-1/2} ds dy_{1}.
\end{equation}

\noindent Again since $|x| (1 - v(s)^{2})^{1/2} \leq  s$ and $g(s) s^{3/2} \in L^{2}(\mathbb{R})$,

\begin{equation}\label{4.15}
|x|^{1/2} (\ref{4.13}) \leq \int_{\sup(\frac{5}{6}, 1 - \frac{|x|}{2T})}^{1} \frac{y_{1}^{3}}{(1 - y_{1}^{2})^{1/2}} \mathcal M(g)(T y_{1}) dy_{1},
\end{equation}

\noindent so by a change of variables,

\begin{equation}\label{4.16}
\Big\| \,|x|^{1/2} (\ref{4.13}) \,\Big\|_{L_{T}^{2} L_{x}^{\infty}(\mathbb{R} \times \mathbb{R}^{4})} \leq \| g \|_{L^{2}(\mathbb{R}^{4})}.
\end{equation}

\noindent Now take $(\ref{4.14})$.

\begin{equation}\label{4.17}
v(s) = \frac{s^{2} - |x|^{2} - T^{2} y_{1}^{2}}{2T |x| y_{1}},
\end{equation}

\noindent so

\begin{align}\label{4.18}
\frac{1}{(1 + v)^{1/2}} =& \frac{(2T |x| y_{1})^{1/2}}{(s^{2} - (|x| - Ty_{1})^{2})^{1/2}}, \notag\\
 \intertext{and}  
 \frac{1}{(1 - v)^{1/2}} =& \frac{(2 T |x| y_{1})^{1/2}}{(s^{2} - (|x| + Ty_{1})^{2})^{1/2}}.
\end{align}

\noindent Therefore, for any $s$ lying in $[\sup(|x| - T, T - 2|x|), |x| + Ty_{1}]$, $|v| \leq 1$, which implies

\begin{equation}
\frac{v}{(1 - v^{2})^{1/2}} \leq \frac{1}{(1 - v)^{1/2}} + \frac{1}{(1 + v)^{1/2}}.
\end{equation}

\noindent Now since $||x| v + T y_{1}| \leq s$, $g(s) s (|x| v(s) + Ty_{1})^{1/2} \in L^{2}(\mathbb{R})$. Therefore, changing the order of integration, since $y_{1} \geq 1 - \frac{|x|}{2T}$,

\begin{equation}\label{4.20}
\aligned
|x|^{1/2} (\ref{4.14}) \leq \frac{1}{|x|^{1/2}} (\sup_{R > 0} \frac{1}{R} \int_{T - R}^{T + R} |g(s) s^{3/2}| ds) \cdot \\ \sup_{s} s^{1/2} \int_{\sup(\frac{5}{6}, 1 - \frac{|x|}{T})}^{1} \frac{y_{1}^{3}}{(1 - y_{1}^{2})^{1/2}} \\ \times \left(\frac{(2T|x| y_{1})^{1/2}}{(s^{2} - (|x| - Ty_{1})^{2})^{1/2}} + \frac{(2 T |x| y_{1})^{1/2}}{(s^{2} - (|x| + Ty_{1})^{2})^{1/2}} \right) dy_{1}.
\endaligned
\end{equation}

\noindent Since

\begin{align}\label{4.19}
&\int_{\sup(\frac{5}{6}, 1 - \frac{|x|}{T})}^{1} \frac{y_{1}^{3}}{(1 - y_{1}^{2})^{1/2}} \left(\frac{(2T|x| y_{1})^{1/2}}{(s^{2} - (|x| - Ty_{1})^{2})^{1/2}} + \frac{(2 T |x| y_{1})^{1/2}}{(s^{2} - (|x| + Ty_{1})^{2})^{1/2}} \right) dy_{1} \notag \\
&\leq \frac{|x|^{1/2}}{s^{1/2}},
\end{align}

\begin{equation}
|x|^{1/2} (\ref{4.14}) \leq \mathcal M(g)(T),
\end{equation}

\noindent and the estimate of the first term in $(\ref{4.11})$ is complete.\vspace{5mm}

\noindent Now consider the second term in $(\ref{4.11})$. Integrating by parts,

\begin{align}
& T^{2} \int_{|y| < \sup(\frac{5}{6}, 1 - \frac{|x|}{2T})} \frac{y \cdot (\nabla g)(x + Ty)}{(1 - |y|^{2})^{1/2}} dy \notag \\
&= T \cdot \sup\left(\frac{5}{6}, 1 - \frac{|x|}{2T}\right) \int_{|y| = \sup(\frac{5}{6}, 1 - \frac{|x|}{T})} \frac{g(x + Ty)}{(1 - |y|^{2})^{1/2}} dy \label{4.21} \\
&- 4T\int_{|y| < \sup(\frac{5}{6}, 1 - \frac{|x|}{2T})} \frac{g(x + Ty)}{(1 - |y|^{2})^{1/2}} dy +
 T \int_{|y| < \sup(\frac{5}{6}, 1 - \frac{|x|}{2T})} \frac{g(x + Ty) |y|^{2} }{(1 - |y|^{2})^{3/2}} dy\label{4.22}
\end{align}

\noindent As in $(\ref{4.7})$, since $g$ is radially symmetric, $|x| \leq 32T$, $g(s) s^{3/2} \in L^{2}([0, \infty))$, $1 - v^{2} \leq 1$, and  $(\ref{variable})$,

\begin{equation}\label{4.23}
|x|^{1/2} (\ref{4.21}) \leq \frac{T^{1/2}}{|x|} \int_{||x| - T \sup(\frac{5}{6}, 1 - \frac{|x|}{2T})|}^{|x| + T \sup(\frac{5}{6}, 1 - \frac{|x|}{2T})} g(s) s (1 - v^{2})^{1/2} ds \leq \mathcal M(g)(T).
\end{equation}

\noindent Also,

\begin{align}\label{4.24}
|x|^{1/2} (\ref{4.22}) \leq& \int_{0}^{\sup(\frac{5}{6}, 1 - \frac{|x|}{2T})} \frac{y_{1}^{3}}{(1 - y_{1}^{2})^{3/2}} \frac{1}{|x|^{1/2}} \int_{||x| - Ty_{1}|}^{|x| + Ty_{1}} g(s) s (1 - v^{2})^{1/2} ds dy_{1} \notag\\
\leq& \int_{0}^{\sup(\frac{5}{6}, 1 - \frac{|x|}{2T})} \frac{y_{1}^{3}}{(1 - y_{1}^{2})^{3/2}} \inf(\frac{|x|^{1/2}}{T^{1/2} y_{1}^{1/2}}, 1) \mathcal M(g)(Ty_{1}) dy_{1}.
\end{align}

\noindent Indeed, since $|x|^{1/2} (1 - v^{2}) \leq s^{1/2}$ and $g(s) s^{3/2} \in L^{2}(\mathbb{R})$, for any $y_{1}$,

\begin{align}
 \frac{y_{1}^{3}}{(1 - y_{1}^{2})^{3/2}} \frac{1}{|x|^{1/2}} \int_{||x| - Ty_{1}|}^{|x| + Ty_{1}} g(s) s (1 - v^{2})^{1/2} ds dy_{1}
\leq  \frac{y_{1}^{3}}{(1 - y_{1}^{2})^{3/2}} \mathcal M(g)(Ty_{1}) dy_{1}.
\end{align}

\noindent Next, when $T y_{1} >> |x|$, $(1 - v^{2}) \leq 1$ implies $|g(s) s (1 - v^{2})^{1/2}| \leq \frac{1}{T^{1/2} y_{1}^{1/2}} |g(s) s^{3/2}|$ for all $s \in [Ty_{1} - |x|, Ty_{1} + |x|]$, which implies

\begin{align}
 &\frac{y_{1}^{3}}{(1 - y_{1}^{2})^{3/2}} \frac{1}{|x|^{1/2}} \int_{||x| - Ty_{1}|}^{|x| + Ty_{1}} g(s) s (1 - v^{2})^{1/2} ds dy_{1} \\
&\leq  \frac{|x|^{1/2}}{T^{1/2} y_{1}^{1/2}} \frac{y_{1}^{3}}{(1 - y_{1}^{2})^{3/2}} \mathcal M(g)(Ty_{1}) dy_{1}.
\end{align}

\begin{equation}\label{4.25}
\Big\|\, |x|^{1/2} (\ref{4.22})\, \Big\|_{L_{T}^{2} L_{x}^{\infty}(\mathbb{R} \times \mathbb{R}^{4})} \leq \| g \|_{L^{2}(\mathbb{R}^{4})}.
\end{equation}

\noindent This completes the proof of the theorem when $v_{0} = 0$. Now suppose $v_{0} = f \in \dot{H}^{1}(\mathbb{R}^{4})$ is radial and $v_{1} = 0$. Then

\begin{equation}\label{4.26}
v(T,x) = 5T \int_{|y| < 1} \frac{y \cdot (\nabla f)(x + Ty)}{(1 - |y|^{2})^{1/2}} dy + T^{2} \int_{|y| < 1} \frac{y_{j} y_{k} (\partial_{j} \partial_{k} f)(x + Ty)}{(1 - |y|^{2})^{1/2}} dy
\end{equation}

\begin{equation}\label{4.27}
+ 3 \int_{|y| < 1} \frac{f(x + Ty)}{(1 - |y|^{2})^{1/2}} dy.
\end{equation}

\noindent Because $f \in \dot{H}^{1}(\mathbb{R}^{4})$, $(\ref{4.26})$ can be estimated in exactly the same manner as $(\ref{4.3})$. This leaves only $(\ref{4.27})$. Since $f$ is radial, making a change of variables,

\begin{equation}\label{4.28}
(\ref{4.27}) = \frac{3}{|x| T} \omega \int_{0}^{1} \frac{y_{1}^{2}}{(1 - y_{1}^{2})^{1/2}} \int_{||x| - Ty_{1}|}^{|x| + Ty_{1}} f(s) s (1 - v^{2})^{1/2} ds dy_{1},
\end{equation}

\noindent so for $|x| \leq 32 T$, since $|x|^{1/2} (1 - v^{2})^{1/2} \leq s^{1/2}$ and $s \sim T$, then by Hardy's inequality

\begin{equation}\label{4.29}
|x|^{1/2} (\ref{4.27}) \leq \int_{0}^{1} \frac{y_{1}^{2}}{(1 - y_{1}^{2})^{1/2}} \mathcal M(f(s) s^{1/2})(T y_{1}) dy_{1}.
\end{equation}

\noindent Then by Hardy's inequality, $f \in \dot{H}^{1}(\mathbb{R}^{4})$ implies $f(s) s^{1/2} \in L^{2}(\mathbb{R})$, so in this case

\begin{equation}\label{4.30}
\Bigg\| \,\sup_{|x| \leq 32T} (|x|^{1/2} v(T,x))\,\Bigg \|_{L_{T}^{2}(\mathbb{R})} \leq \| f \|_{\dot{H}^{1}(\mathbb{R}^{4})}.
\end{equation}

\noindent This completes the proof of the theorem. $\Box$

\subsection{Inhomogeneous wave equation estimate}

\begin{theorem}\label{t5}
Suppose that $v$ is a solution to the wave equation

\begin{equation}\label{5.1}
\left. \begin{array}{rcl}
\leftexp{4+1}{\square}\, v &=& F(v)\,\,\,\,\,\,\,\,\,\,\,\,\,\,\, \,\,\,\,\,\,\,\,\,\,\,\,\,\,\,\,\,\,\textnormal{on}\,\, \mathbb{R}^{4+1}\\
 v_0 = v (0, x) & \textnormal{and}& v_1 = \ptl_T v (0, x) \,\,\,\,\,\,\,\, \textnormal{on}\,\, \mathbb{R}^4\\\end{array} 
\right\}
\end{equation}

\noindent Then it follows that

\begin{align}\label{5.4}
&\Big\| \, v\,\Big\|_{L_{T}^{2} L_{x}^{8}(\mathbb{R} \times \mathbb{R}^{4})} + \Big\| |x|^{1/4} v \Big\|_{L_{T}^{2} L_{x}^{16}(\mathbb{R} \times \mathbb{R}^{4})} \notag \\ 
&\leq \| v_{0} \|_{\dot{H}^{1}(\mathbb{R}^{4})} + \| v_{1} \|_{L^{2}(\mathbb{R}^{4})} +\left (\sum_{j} 2^{j/2} \| F \|_{L_{T,x}^{2}(\mathbb{R} \times \{ x : 2^{j} \leq |x| \leq 2^{j + 1} \})}\right)
\end{align}
\end{theorem}

\noindent \emph{Proof:}  It suffices to prove $(\ref{5.4})$ for $F \in L_{T,x}^{2}(\mathbb{R} \times \mathbb{R}^{4})$ when $F$ is supported on $\frac{\rho}{2} \leq |x| \leq \rho$ with bounds independent of $\rho$. By finite propagation speed, $\frac{\sin((T - \tau) \sqrt{-\Delta})}{\sqrt{-\Delta}} F(\tau)$ is supported on $|x| \leq \rho + |T - \tau|$.\vspace{5mm}

\noindent For $|x| < |T - \tau| - 3\rho$, observe that by the fundamental solution of the wave equation (see for example Sogge\cite{sogge_book})

\begin{align}\label{5.15}
\frac{\sin(T - \tau) \sqrt{-\Delta}}{\sqrt{-\Delta}} F(\tau) =& 3(T - \tau) \int_{|y| < 1} \frac{F(\tau, x + (T - \tau) y)}{(1 - |y|^{2})^{1/2}} dy \notag
\\ &\quad + (T - \tau)^{2} \int_{|y| < 1} \frac{\nabla F(\tau, x + (T - \tau) y) \cdot y}{(1 - |y|^{2})^{1/2}} dy.
\end{align}

\noindent If $|x| < |T - \tau| - 3\rho$ and $F(\tau, z)$ is supported on $|z| \leq \rho$, then for $|T- \tau| > 3 \rho$, $|y| < 1 - \frac{2\rho}{|T - \tau|}$. Then by H{\"o}lder's inequality and change of variables,

\begin{align}\label{5.16}
&\Bigg\| 3(T - \tau) \int_{|y| < 1} \frac{F(\tau, x + (T - \tau) y)}{(1 - |y|^{2})^{1/2}} dy \Bigg\|_{L_{x}^{\infty}(|x| < |T - \tau| - 3\rho)} \notag\\ 
&\leq\frac{\| F(\tau) \|_{L_{x}^{1}(\mathbb{R}^{4})}}{|T - \tau|^{5/2} \rho^{1/2}} \leq \frac{\rho^{3/2}}{|T - \tau|^{5/2}} \| F(\tau) \|_{L_{x}^{2}(\mathbb{R}^{4})}.
\end{align}

\noindent Also, integrating by parts,

\begin{align}\label{5.17}
&\Bigg\| (T - \tau)^{2} \int_{|y| < 1} \frac{\nabla F(\tau, x + (T - \tau) y) \cdot y}{(1 - |y|^{2})^{1/2}} dy \Bigg\|_{L_{x}^{\infty}(|x| < |T - \tau| - 3\rho)} \notag\\ 
&\leq \left(\frac{1}{\rho^{3/2} |T - \tau|^{3/2}} + \frac{1}{\rho^{1/2} |T - \tau|^{5/2}}\right) \| F(\tau) \|_{L_{x}^{1}(\mathbb{R}^{4})} \notag\\ 
&\leq \left(\frac{\rho^{1/2}}{|T - \tau|^{3/2}} + \frac{\rho^{3/2}}{|T - \tau|^{5/2}}\right) \| F(\tau) \|_{L_{x}^{2}(\mathbb{R}^{4})}.
\end{align}

\noindent Meanwhile, by the Sobolev embedding theorem and H{\"o}lder's inequality,

\begin{equation}\label{5.18}
\Bigg\Vert \frac{\sin((T - \tau) \sqrt{-\Delta})}{\sqrt{-\Delta}} F(\tau) \Bigg\Vert_{L_{x}^{2}(\mathbb{R}^{4})} \leq \| F(\tau) \|_{L_{x}^{4/3}(\mathbb{R}^{4})} \leq \rho \| F(\tau) \|_{L_{x}^{2}(\mathbb{R}^{4})}.
\end{equation}

\noindent Interpolating $(\ref{5.16})$, $(\ref{5.17})$, and $(\ref{5.18})$,

\begin{align}\label{5.19}
&\Bigg\| \frac{\sin((T - \tau) \sqrt{-\Delta})}{\sqrt{-\Delta}} F(\tau) \Bigg\|_{L_{x}^{8}(|x| \leq |T - \tau| - 3\rho)} \notag\\
\leq& \left(\frac{\rho^{5/8}}{|T - \tau|^{9/8}} + \frac{\rho^{11/8}}{|T - \tau|^{15/8}}\right) \Big\| F(\tau) \Big\|_{L_{x}^{2}(\mathbb{R}^{4})}.
\end{align}

\noindent \textbf{Remark:} If we were in odd dimensions the sharp Huygens principle would imply that $(\ref{5.19})$ is identically zero. However, since we are in even dimensions, $(\ref{5.19})$ is nonzero.\vspace{5mm}

\noindent Next, by finite propagation speed and interpolating $(\ref{5.16})$, $(\ref{5.17})$, and $(\ref{5.19})$,

\begin{align}\label{5.20}
&\Bigg\| |x|^{1/4} \frac{\sin((T - \tau) \sqrt{-\Delta})}{\sqrt{-\Delta}} F(\tau) \Bigg\|_{L_{x}^{16}(|x| \leq |T - \tau| - 3 \rho)}  \notag \\ 
\leq & \left(\frac{\rho^{9/16}}{|T- \tau|^{17/16}} + \frac{\rho^{25/16}}{|T - \tau|^{33/16}}\right) \Big\| F(\tau) \Big\|_{L_{x}^{2}(\mathbb{R}^{4})}.
\end{align}

\noindent Therefore, for any $l \in \mathbb{Z}$, $l \geq 0$, if $\tau \in [(l - 1) \rho, l \rho]$, $|t - \tau| > \rho$,

\begin{align}\label{5.21}
&\Bigg\| \frac{\sin((T - \tau) \sqrt{-\Delta})}{\sqrt{-\Delta}} F(\tau) \Bigg\|_{L_{x}^{8}(\mathbb{R}^{4})} +\Bigg \| |x|^{1/4} \frac{\sin((T - \tau) \sqrt{-\Delta})}{\sqrt{-\Delta}} F(\tau) \Bigg\|_{L_{x}^{16}(\mathbb{R} \times \mathbb{R}^{4})} \notag \\
&\leq \Bigg\| \frac{\sin((T - \tau) \sqrt{-\Delta})}{\sqrt{-\Delta}} F(\tau) \Bigg\|_{L_{x}^{8}\left((l - 1)\rho \leq T - |x| \leq (l + 4) \rho \right)} \notag\\
&\quad + \Bigg\| |x|^{1/4} \frac{\sin((T - \tau) \sqrt{-\Delta})}{\sqrt{-\Delta}} F(\tau) \Bigg\|_{L_{x}^{16}((l - 1)\rho \leq T - |x| \leq (l + 4) \rho)}  \notag \\ 
& \quad + \frac{\rho^{9/16}}{|T - \tau|^{17/16}} \Big\| \,F(\tau)\, \Big\|_{L_{x}^{2}(\mathbb{R}^{4})}.
\end{align}
\noindent Now we abuse notation and let $\lfloor T - \rho \rfloor = \lfloor \frac{T - \rho}{\rho} \rfloor \rho$, where $\lfloor x \rfloor$ is the integer part of $x$. Then because the sets $\{ (t, x) : (l - 1) \rho \leq t - |x| \leq (l + 4) \rho \}$ are pairwise disjoint for any two $l_{1}, l_{2} \in \mathbb{Z}$,

\begin{align}\label{5.23}
&\Bigg\| \int_{0}^{\lfloor T - \rho \rfloor} \frac{\sin((T - \tau) \sqrt{-\Delta})}{\sqrt{-\Delta}} F(\tau) d\tau \Bigg\|_{L_{T}^{2} L_{x}^{8}(\mathbb{R} \times \mathbb{R}^{4})} \notag\\ 
&\leq\Bigg \| \int_{0}^{T - \rho} \frac{\rho^{9/16}}{|T - \tau|^{17/16}}  \| F(\tau) \|_{L_{x}^{2}(\mathbb{R}^{4})} d\tau \Bigg\|_{L_{T}^{2}(\mathbb{R})} \notag\\
&\quad + \left(\sum_{l \in \mathbb{Z}} \Bigg\| \int_{l\rho}^{(l + 1)\rho} \frac{\sin((T - \tau) \sqrt{-\Delta})}{\sqrt{-\Delta}} F(\tau) d\tau \Bigg\|_{L_{T}^{2} L_{x}^{8}(\mathbb{R} \times \mathbb{R}^{4})}^{2}\right)^{1/2}.
\end{align}

\noindent Then by Theorem $\ref{t1}$, H{\"o}lder's inequality, and Young's inequality,

\begin{align}\label{5.24}
\leq & \rho^{1/2} \Big\|\, F\, \Big\|_{L_{T,x}^{2}(\mathbb{R} \times \mathbb{R}^{4})} + \left(\sum_{l \in \mathbb{Z}} \| F(\tau) \|_{L_{T}^{1} L_{x}^{2}([l \rho, (l + 1) \rho] \times \mathbb{R}^{4})}^{2}\right)^{1/2} \notag\\ \leq &\rho^{1/2} \Big\|\, F\, \Big\|_{L_{T,x}^{2}(\mathbb{R} \times \mathbb{R}^{4})}.
\end{align}

\noindent Finally, by Theorem $\ref{t1}$ and Young's inequality,

\begin{equation}\label{5.25}
\Bigg\| \int_{\lfloor T - \rho \rfloor}^{T} \frac{\sin((T - \tau) \sqrt{-\Delta})}{\sqrt{-\Delta}} F(\tau) d\tau \Bigg\|_{L_{T}^{2} L_{x}^{8}(\mathbb{R} \times \mathbb{R}^{4})} \leq \rho^{1/2} \| F \|_{L_{T,x}^{2}(\mathbb{R} \times \mathbb{R}^{4})}.
\end{equation}

\noindent Therefore, Theorem $\ref{t1}$, $(\ref{5.23})$, $(\ref{5.24})$, and $(\ref{5.25})$ combine to prove

\begin{equation}\label{5.26}
\| v \|_{L_{T}^{2} L_{x}^{8}(\mathbb{R} \times \mathbb{R}^{4})} \leq \| v_{0} \|_{\dot{H}^{1}(\mathbb{R}^{4})} + \| v_{1} \|_{L_{x}^{2}(\mathbb{R}^{4})} + \rho^{1/2} \| F \|_{L_{T,x}^{2}(\mathbb{R} \times \mathbb{R}^{4})}.
\end{equation}

\noindent Replacing Theorem $\ref{t1}$ with Theorem $\ref{t4}$ in $(\ref{5.23})$ - $(\ref{5.25})$, proves

\begin{equation}\label{2.25}
\Big\| |x|^{1/4} v \Big\|_{L_{T}^{2} L_{x}^{16}(\mathbb{R} \times \mathbb{R}^{4})} \leq \| v_{0} \|_{\dot{H}^{1}(\mathbb{R}^{4})} + \| v_{1} \|_{L_{x}^{2}(\mathbb{R}^{4})} + \rho^{1/2} \| F \|_{L_{T,x}^{2}(\mathbb{R} \times \mathbb{R}^{4})},
\end{equation}

\noindent and thus completes the proof of Theorem $\ref{t5}$. $\Box$

\subsection{A Function Space}
\noindent We will use the function space

\begin{definition}[Function spaces]\label{d6}
If $P_{N}$ is a Littlewood - Paley operator then let

\begin{align}\label{6.1}
\| v \|_{X}^{2} =& \sum_{N} \Big\| \,P_{N} v \,\Big\|_{L_{T}^{2} L_{x}^{8}(\mathbb{R} \times \mathbb{R}^{4})}^{2} + \sum_{N} \Big\| |x|^{1/4} P_{N} v \Big\|_{L_{t}^{2} L_{x}^{16}(\mathbb{R} \times \mathbb{R}^{4})}^{2} \notag\\
&+ \sum_{N} N^{2} \Big\| P_{N} v \Big\|_{L_{T}^{\infty} L_{x}^{2}(\mathbb{R} \times \mathbb{R}^{4})}^{2} \notag\\
&+ \sum_{N} \left(\sup_{\rho > 0} \rho^{-1/2} \Big\| P_{N} \ptl_T v \Big\|_{L_{T,x}^{2}(\mathbb{R} \times \{ x : |x| \leq \rho \})}\right)^{2}  \notag \\
&+ \sum_{N} \left(\sup_{\rho > 0} \rho^{-1/2} \Big\| P_{N} \nabla_x v \Big\|_{L_{T,x}^{2}(\mathbb{R} \times \{ x : |x| \leq \rho \})}\right)^{2} \notag\\
&+ \sum_{N} N^{2} \left(\sup_{\rho > 0} \rho^{-1/2} \Big\| P_{N} v \Big\|_{L_{T,x}^{2}(\mathbb{R} \times \{ x : |x| \leq \rho \})}\right)^{2}\notag \\
&+ \sum_{N} \Big\| |x|^{-3/2} P_{N} v \Big\|_{L_{T,x}^{2}(\mathbb{R} \times \mathbb{R}^{4})}^{2} 
+ \sum_{N} N^{-2} \Big\| P_{N} \ptl_T v \Big\|_{L_{T}^{2} L_{x}^{8}(\mathbb{R} \times \mathbb{R}^{4})}^{2}.
\end{align}

\noindent We also define the norm

\begin{align} \label{6.2}
\| F \|_{Y}^{2} = & \inf_{F_{1} + F_{2} = F} \Big\| F_{1} \Big\|_{L_{T}^{1} L_{x}^{2}(\mathbb{R} \times \mathbb{R}^{4})}^{2}  \notag \\
& \quad+ \sum_{N} \left(\sum_{j} 2^{j/2} \| P_{N} F_{2} \|_{L_{T,x}^{2}(\mathbb{R} \times \{ 2^{j} \leq |x| \leq 2^{j + 1} \})}\right)^{2}.
\end{align}
\end{definition}

\begin{lemma}\label{l7}

\begin{align}\label{7.2}
& \left(\sup_{\rho > 0} \rho^{-1/2} \Big\| \,\nabla v \,\Big\|_{L_{T,x}^{2}(\mathbb{R} \times \{ x : |x| \leq \rho \})}\right) + \left(\sup_{\rho > 0} \rho^{-1/2} \Big\|\, \ptl_T \, v\Big\|_{L_{T,x}^{2}(\mathbb{R} \times \{ x : |x| \leq \rho \})}\right) \notag \\ & \leq \| v \|_{X}.
\end{align}
\end{lemma}

\noindent \emph{Proof:} Fix $\rho > 0$. Letting

\begin{equation}\label{7.2.1}
\tilde{P}_{N} = P_{\frac{N}{2}} + P_{N} + P_{2N},
\end{equation}

\begin{equation}\label{7.3}
\phi \left(\frac{x}{\rho}\right) \nabla P_{N} v = \tilde{P}_{N} \left( \phi\left(\frac{x}{\rho}\right) \nabla P_{N} v\right) + \left(P_{> \frac{N}{8}} \phi\left(\frac{x}{\rho}\right)\right) \nabla P_{N} v.
\end{equation}

\noindent By Bernstein's inequality,

\begin{equation}\label{7.4}
\Big\| P_{> \frac{N}{8}} \phi\left(\frac{x}{\rho}\right)\Big\|_{L_{T}^{\infty} L_{x}^{8/3}(\mathbb{R} \times \mathbb{R}^{4})} \leq \inf(N^{-2} \rho^{-2}, 1),
\end{equation}

\noindent so

\begin{align}\label{7.5}
&\rho^{-1/2} \Bigg\| \sum_{N} \left(P_{> \frac{N}{8}} \phi\left(\frac{x}{\rho}\right)\right) \nabla P_{N} v\Bigg \|_{L_{T,x}^{2}(\mathbb{R} \times \mathbb{R}^{4})} \notag\\ 
&\leq \rho^{-1/2} \sum_{N} \inf \left(N^{-2} \rho^{-2}, 1\right) N \Big\| P_{N} v \Big\|_{L_{T}^{2} L_{x}^{8}(\mathbb{R} \times \mathbb{R}^{4})} \leq \rho^{-3/2} \| v \|_{X}.
\end{align}

\noindent Meanwhile, since the $\tilde{P}_{N}$ are finitely overlapping,

\begin{align}\label{7.6}
\rho^{-1} \Bigg\| \sum_{N} \tilde{P}_{N} \left(\phi \left(\frac{x}{\rho}\right) \nabla P_{N} v \right) \Bigg\|_{L_{T,x}^{2}(\mathbb{R} \times \mathbb{R}^{4})}^{2} \leq& \rho^{-1} \sum_{N} \Bigg\| \phi\left(\frac{x}{\rho}\right) \nabla P_{N} v \Bigg\|_{L_{T,x}^{2}(\mathbb{R} \times \mathbb{R}^{4})}^{2} \notag\\ 
\leq & \| v \|_{X}^{2},
\end{align}

\noindent which proves

\begin{equation}\label{7.7}
\left(\sup_{\rho > 0} \rho^{-1/2} \Big\| \, \nabla v \,\Big\|_{L_{T,x}^{2}(\mathbb{R} \times \{ x : |x| \leq \rho \})}\right) \leq \| v \|_{X}.
\end{equation}

\noindent The proof of 

\begin{equation}\label{7.8}
\left(\sup_{\rho > 0} \rho^{-1/2} \Big\| \ptl_T v \Big\|_{L_{T,x}^{2}(\mathbb{R} \times \{ x : |x| \leq \rho \})}\right) \leq \| v \|_{X}.
\end{equation}

\noindent is similar. $\Box$

\begin{lemma}
\begin{equation}\label{7.1}
\Big\| |x|^{1/4} v\Big\|_{L_{T}^{2} L_{x}^{16}(\mathbb{R} \times \mathbb{R}^{4})} \leq \| v \|_{X},
\end{equation}
\end{lemma}

\noindent \emph{Proof:} By the Littlewood - Paley theorem

\begin{equation}\label{est0}
\Big\| |x|^{1/4} v \Big\|_{L_{T}^{2} L_{x}^{16}(\mathbb{R} \times \mathbb{R}^{4})}^{2} \leq \sum_{N} \Big\| P_{N}(|x|^{1/4} v) \Big\|_{L_{T}^{2} L_{x}^{16}(\mathbb{R} \times \mathbb{R}^{4})}^{2}.
\end{equation}

\noindent By H{\"o}lder's inequality

\begin{align}\label{est1}
\Big\| |x|^{1/4} \phi(Nx) (P_{\leq N} v) \Big\|_{L_{T}^{2} L_{x}^{16}} \leq & N^{-1/2} \sum_{M \leq N} \big\| P_{M} v \big\|_{L_{T}^{2} L_{x}^{\infty}}  \notag \\ \leq& \left(\frac{M}{N}\right)^{1/2} \sum_{M \leq N} \big\| P_{M} v \big\|_{L_{T}^{2} L_{x}^{8}}.
\end{align}

\noindent By the Sobolev embedding theorem $\dot{H}^{7/4}(\mathbb{R}^{4}) \subset L^{16}(\mathbb{R}^{4})$,

\begin{equation}\label{est2}
\aligned
\Big\| P_{N}(|x|^{1/4} \phi(Nx) (P_{> N} v)) \Big\|_{L_{T}^{2} L_{x}^{16}} \leq N^{3/2} \sum_{M \geq N} \big\| P_{M} v \big\|_{L_{T}^{2} L_{x}^{2}(|x| \leq \frac{2}{N})} \\
\leq \sum_{M \geq N} \left(\frac{N}{M}\right) M \left(\sup_{\rho > 0} \rho^{-1/2} \big\| P_{M} v \big\|_{L_{T}^{2} L_{x}^{2}(|x| \leq \rho)}\right).
\endaligned
\end{equation}

\noindent Next,

\begin{equation}\label{est3}
\left\| |x|^{1/4} (1 - \phi(Nx)) (\tilde{P}_{N} v) \right\|_{L_{T}^{2} L_{x}^{16}} \leq \sum_{\frac{N}{32} \leq M \leq 32 N} \left\| |x|^{1/4} P_{M} v \right\|_{L_{T}^{2} L_{x}^{16}}.
\end{equation}

\noindent Next, by Bernstein's inequality, letting $\chi(x) = \phi(\frac{x}{2}) - \phi(x)$,

\begin{equation}\label{est4}
\aligned
\sum_{j \geq 0} \left\| P_{N} \left( |x|^{1/4} \chi(2^{-j} N x) (P_{> 32N} v) \right) \right\|_{L_{T}^{2} L_{x}^{16}} \\ \leq \sum_{j \geq 0} N^{3/4} \left\| \nabla \left(|x|^{1/4} \chi(2^{-j} Nx) \right) \right\|_{L_{T}^{\infty} L_{x}^{\infty}} \left\| P_{> 32N} v \right\|_{L_{T}^{2} L_{x}^{2}(|x| \leq \frac{2^{j}}{N})} \\ \leq \sum_{j \geq 0} 2^{-3j/4} N^{3/2} \sum_{M \geq 32N} \frac{2^{j/2}}{M N^{1/2}} \cdot \sup_{\rho > 0} \left( \rho^{-1/2} M \| P_{M} v \|_{L_{T}^{2} L_{x}^{2}(|x| \leq \rho)} \right) \\ \leq \sum_{M \geq 32N} \left(\frac{N}{M}\right) \cdot \sup_{\rho > 0} \left( \rho^{-1/2} M \| P_{M} v \|_{L_{T}^{2} L_{x}^{2}(|x| \leq \rho)} \right).
\endaligned
\end{equation}

\noindent Finally, by H\"older's inequality and Sobolev embedding,

\begin{equation}\label{est5}
\aligned
\sum_{j \geq 0} \left\| P_{N} \left( |x|^{1/4} \chi(2^{-j} N x) (P_{\leq \frac{N}{32}} v) \right) \right\|_{L_{T}^{2} L_{x}^{16}} \\ \leq \sum_{j \geq 0} \frac{1}{N} \left\| \nabla \left(|x|^{1/4} \chi(2^{-j} Nx) \right) \right\|_{L_{T}^{\infty} L_{x}^{\infty}} \left\| P_{\leq \frac{N}{32}} v \right\|_{L_{T}^{2} L_{x}^{16}(|x| \leq \frac{2^{j}}{N})} \\ \leq \sum_{j \geq 0} 2^{-3j/4} N^{-1/4} \sum_{M \leq \frac{N}{32}} \frac{2^{j/4}}{N^{1/4}} \left\| P_{M} v \right\|_{L_{T}^{2} L_{x}^{\infty}} \leq \sum_{M \leq \frac{N}{32}} \left(\frac{M}{N}\right)^{1/2} \left\| P_{M} v \right\|_{L_{T}^{2} L_{x}^{8}}.
\endaligned
\end{equation}

\noindent Combining $(\ref{est1})$ - $(\ref{est5})$, by Young's inequality and $(\ref{6.1})$,

\begin{equation}
\sum_{N} \left\| P_{N}(|x|^{1/4} v) \right\|_{L_{T}^{2} L_{x}^{16}(\mathbb{R} \times \mathbb{R}^{4})}^{2} \leq \| v \|_{X}^{2}.
\end{equation}

\noindent $\Box$

\begin{theorem}\label{t8}
If $v$ is a radial solution to the equation
\begin{equation}\label{8.1}
\left. \begin{array}{rcl}
\leftexp{4+1}{\square}\, v &=& F(v)\,\,\,\,\,\,\,\,\,\,\,\,\,\,\, \,\,\,\,\,\,\,\,\,\,\,\,\,\,\,\,\,\,\textnormal{on}\,\, \mathbb{R}^{4+1}\\
 v_0 = v (0, x) & \textnormal{and}& v_1 = \ptl_T v (0, x) \,\,\,\,\,\,\,\, \textnormal{on}\,\, \mathbb{R}^4\\\end{array} 
\right\}
\end{equation}

\noindent then

\begin{equation}\label{8.2}
\| v \|_{X} \leq \| v_{0} \|_{\dot{H}^{1}(\mathbb{R}^{4})} + \| v_{1} \|_{L^{2}(\mathbb{R}^{4})} + \| F \|_{Y}.
\end{equation}
\end{theorem}

\noindent \emph{Proof:} By Theorem $\ref{t5}$, it only remains to prove

\begin{equation}\label{8.3}
\sum_{N} N^{2} \sup_{\rho > 0} \rho^{-1/2} \Big\| P_{N} v \Big\|_{L_{T,x}^{2}(\mathbb{R} \times \mathbb{R}^{4})}^{2} \leq  \| v_{0} \|_{\dot{H}^{1}(\mathbb{R}^{4})} + \| v_{1} \|_{L^{2}(\mathbb{R}^{4})} + \| F \|_{Y}.
\end{equation}

\noindent and

\begin{equation}\label{8.3.1}
\sum_{N} N^{-2} \Big\| P_{N} \ptl_T v \Big\|_{L_{T}^{2} L_{x}^{8}(\mathbb{R} \times \mathbb{R}^{4})}^{2} \leq  \| v_{0} \|_{\dot{H}^{1}(\mathbb{R}^{4})} + \| v_{1} \|_{L^{2}(\mathbb{R}^{4})} + \| F \|_{Y}.
\end{equation}

\noindent We start with $(\ref{8.3.1})$. Fix $N$. Suppose $\phi \in C_{0}^{\infty}(\mathbb{R}^{4})$ is a positive radial function, $\phi(x) = 1$ for $|x| \leq 1$, $\phi(x) = 0$ for $|x| > 2$. If $\rho \geq \frac{1}{N}$, we take the commutator

\begin{align}\label{8.4}
\rho^{-1/2} N \phi\left(\frac{x}{\rho}\right) P_{N} v =& \rho^{-1/2} N \phi\left(\frac{x}{\rho}\right) P_{N} \tilde{P}_{N} v \notag\\
=& \rho^{-1/2} N P_{N} \left(\phi\left(\frac{x}{\rho}\right) \tilde{P}_{N} v\right) + \rho^{-1/2} N \left[\phi\left(\frac{x}{\rho}\right), P_{N}\right] \tilde{P}_{N} v.
\end{align}

\noindent Then by Bernstein's inequality

\begin{align}\label{8.5}
&\rho^{-1/2} \Big\| N P_{N} \left(\phi\left(\frac{x}{\rho}\right) \tilde{P}_{N} v \right) \Big\|_{L_{T,x}^{2}(\mathbb{R} \times \mathbb{R}^{4})} \notag\\ 
\leq & \,\rho^{-1/2} \Big\| \nabla \left( \phi\left(\frac{x}{\rho}\right) \tilde{P}_{N} v \right) \Big\|_{L_{T,x}^{2}(\mathbb{R} \times \mathbb{R}^{4})} \notag\\
\leq &\, \rho^{-3/2} \Big\| \phi' \left(\frac{x}{\rho}\right) \tilde{P}_{N} v \Big\|_{L_{T,x}^{2}(\mathbb{R} \times \mathbb{R}^{4})} + \rho^{-1/2} \Big\| \phi \left(\frac{x}{\rho}\right) \nabla \tilde{P}_{N} v \Big\|_{L_{T,x}^{2}(\mathbb{R} \times \mathbb{R}^{4})} \notag\\
\leq& \, \Big\| \tilde{P}_{N} v \Big\|_{L_{T}^{2} L_{x}^{8}(\mathbb{R} \times \mathbb{R}^{4})} + \rho^{-1/2} \Big\| \phi\left(\frac{x}{\rho}\right) \nabla \tilde{P}_{N} v \Big\|_{L_{T,x}^{2}(\mathbb{R} \times \mathbb{R}^{4})}.
\end{align}

\noindent Now compute the commutator

\begin{equation}\label{8.5.1}
\rho^{-1/2} N[P_{N}, \phi(\frac{x}{\rho})] \tilde{P}_{N} v = \rho^{-1/2} N \int N^{4} K(N(x - y)) [\phi(\frac{x}{\rho}) - \phi(\frac{y}{\rho})] (\tilde{P}_{N} v)(y) dy,
\end{equation}

\noindent where $K(\cdot)$ is the kernel of the Littlewood - Paley projection $P_{1}$. By the fundamental theorem of calculus, $|\phi(\frac{x}{\rho}) - \phi(\frac{y}{\rho})| \leq \frac{|x - y|}{\rho}$, so by H{\"o}lder's inequality, because $K(\cdot)$ is rapidly decreasing for $|x| \geq 1$,

\begin{equation}\label{8.5.2}
\| \rho^{-1/2} N \int N^{4} K(N(x - y)) [\phi(\frac{x}{\rho}) - \phi(\frac{y}{\rho})] (\tilde{P}_{N} v)(y) dy \|_{L_{x}^{2}(|x| \leq 10 \rho)} \leq \| \tilde{P}_{N} v \|_{L_{x}^{8}}.
\end{equation}

\noindent When $|x| \geq 10 \rho$, by the support of $\phi$ and the fact that $K(\cdot)$ is rapidly decreasing for $|x| \geq 1$,

\begin{equation}\label{8.5.3}
\aligned
\rho^{-1/2} N \int N^{4} K(N(x - y)) [\phi(\frac{x}{\rho}) - \phi(\frac{y}{\rho})] (\tilde{P}_{N} v)(y) dy \\= -\rho^{-1/2} N \int N^{4} K(N(x - y)) \phi(\frac{y}{\rho}) (\tilde{P}_{N} v)(y) dy \\
\leq \rho^{-1/2} \frac{N}{N^{10} |x|^{10}} \int N^{4} K(N(x - y)) \phi(\frac{y}{\rho}) (\tilde{P}_{N} v)(y) dy.
\endaligned
\end{equation}

\noindent Therefore, if $\rho \geq \frac{1}{N}$,

\begin{equation}\label{8.6}
N \rho^{-1/2} \Bigg\| \left[\phi\left(\frac{x}{\rho}\right), P_{N}\right] \tilde{P}_{N} v \Bigg\|_{L_{T,x}^{2}(\mathbb{R} \times \mathbb{R}^{4})} \leq \Big\| \tilde{P}_{N} v \Big\|_{L_{T}^{2} L_{x}^{8}(\mathbb{R} \times \mathbb{R}^{4})}.
\end{equation}

\noindent On the other hand, if $\rho \leq \frac{1}{N}$, then simply apply Holder's inequality,

\begin{equation}\label{8.7}
\rho^{-1/2} N \Bigg\| \phi\left(\frac{x}{\rho}\right) P_{N} v \Bigg\|_{L_{T,x}^{2}(\mathbb{R} \times \mathbb{R}^{4})} \leq \Big\| \tilde{P}_{N} v \Big\|_{L_{T}^{2} L_{x}^{8}(\mathbb{R} \times \mathbb{R}^{4})}.
\end{equation}

\noindent Combining $(\ref{8.5})$, $(\ref{8.6})$, and $(\ref{8.7})$,

\begin{align}\label{8.8}
&\sum_{N} \left(\sup_{\rho > 0} \rho^{-1/2} N \| P_{N} v \|_{L_{T,x}^{2}(\mathbb{R} \times \{ x : |x| \leq \rho\})} \right)^{2} \notag\\ 
&\leq \sum_{N} \left(\sup_{\rho > 0} \rho^{-1/2} \Big\| \nabla P_{N} v \Big \|_{L_{T,x}^{2}(\mathbb{R} \times \{ x : |x| \leq \rho \})}\right)^{2} + \sum_{N} \Big\| P_{N} v \Big\|_{L_{T}^{2} L_{x}^{8}(\mathbb{R} \times \mathbb{R}^{4})}^{2}.
\end{align}

\noindent The proof of $(\ref{8.3.1})$ is straightforward. Applying $(\ref{5.15})$, the Huygens principle, Theorem $\ref{t1}$ and $\ref{t4}$,

\begin{align}\label{8.9}
&\Big\| P_{N} \ptl_T v \Big\|_{L_{T}^{2} L_{x}^{8}(\mathbb{R} \times \mathbb{R}^{4})} \notag\\
&\leq N \Big\| P_{N} v_{0} \Big\|_{\dot{H}^{1}(\mathbb{R}^{4})} + N \Big\| P_{N} v_{1} \Big\|_{L^{2}(\mathbb{R}^{4})} + \Big\| \nabla P_{N} F \Big\|_{Y + L_{T}^{2} L_{x}^{8/7}},
\end{align}

\noindent where if $V$ and $W$ are Banach spaces, 
\[ \Big\| \, f \,\Big\|_{V + W} = \inf_{f = f_{1} + f_{2}} \Big\|\, f_{1} \,\Big\|_{V} + \Big\|\, f_{2} \,\Big\|_{W}. \]
\vspace{5mm}

\noindent By Holder's inequality,

\begin{equation}\label{8.10}
\Big\| \nabla P_{N} (\phi(Nx) \tilde{P}_{N} F) \Big\|_{L_{tT}^{2} L_{x}^{8/7}(\mathbb{R} \times \mathbb{R}^{4})} \leq N \sum_{j} 2^{j/2} \Big\| \tilde{P}_{N} F \Big\|_{L_{T,x}^{2}(\mathbb{R} \times \{ x : 2^{j} \leq |x| \leq 2^{j + 1} \})}.
\end{equation}

\noindent Meanwhile, by Holder's inequality and the fact that the Littlewood - Paley kernel is rapidly decreasing,

\begin{align}\label{8.11}
\Big\| (1 - \phi(Nx)) \nabla& P_{N} ((1 - \phi(Nx)) \tilde{P}_{N} F) \Big\|_{Y} \notag\\
&+ \Big\| \phi(Nx) \nabla P_{N} ((1 - \phi(Nx)) \tilde{P}_{N} F) \Big\|_{L_{T}^{2} L_{x}^{8/7}(\mathbb{R} \times \mathbb{R}^{4})} \notag \\
&\quad\quad\leq N \sum_{j} 2^{j/2} \Big\| \tilde{P}_{N} F \Big\|_{L_{T,x}^{2}(\mathbb{R} \times \{ x : 2^{j} \leq |x| \leq 2^{j + 1} \})}.
\end{align}

\noindent Since $\tilde{P}_{N}$ has finite overlap, this completes the proof of $(\ref{8.3.1})$. $\Box$

\subsection{Scattering for Small Data}

\begin{theorem}[Scattering]\label{t9}
The nonlinear wave equation
\begin{equation}\label{9.1}
\left. \begin{array}{rcl}
\leftexp{4+1}{\square}\, v &=& F(v)\,\,\,\,\,\,\,\,\,\,\,\,\,\,\, \,\,\,\,\,\,\,\,\,\,\,\,\,\,\,\,\,\,\textnormal{on}\,\, \mathbb{R}^{4+1}\\
 v_0 = v (0, x) & \textnormal{and}& v_1 = \ptl_T v (0, x) \,\,\,\,\,\,\,\, \textnormal{on}\,\, \mathbb{R}^4\\\end{array} 
\right\}
\end{equation}
with 
\[F(v) = \left(e^{2Z} - 1 + \left(\frac{r}{R}\partial_{\eta} r + \frac{1}{2}\right) - \left(\frac{r}{R}\partial_{\xi} r - \frac{1}{2}\right)\right) \frac{v}{r^{2}} + e^{2Z} \frac{R^{2}}{r^{2}} v^{3} \zeta(R v) \]

and
\begin{align} \label{metric-small}
\Big|e^{2Z} - 1 \Big|, \hspace{5mm} \Big|\frac{R}{r} - 1\Big|, \hspace{5mm} \Big|R v(T, R)\Big| \leq E(v)
\end{align}

\noindent has a solution with $\| v \|_{L_{T}^{2} L_{x}^{8}(\mathbb{R} \times \mathbb{R}^{4})} < \infty$ for energy $E(v)$ sufficiently small.
\end{theorem}

\noindent \emph{Proof:} By Theorem $\ref{t8}$ and \eqref{metric-small}, it suffices to prove

\begin{equation}\label{9.2}
\aligned
\Bigg\| \left(e^{2Z} - 1 + \left(\partial_{\eta} r + \frac{1}{2}\right) - \left(\partial_{ \xi} r - \frac{1}{2}\right) \right) \frac{v}{R^{2}} + e^{2Z} \frac{R^{2}}{r^{2}} \zeta(R v) v^{3} \Bigg\|_{Y} \\ \leq \| v \|_{X}^{2} E(v)^{1/2} + E(v)^{3} + c(E(v)) \| v \|_{X},
\endaligned
\end{equation}

\noindent for some quantity $c(E(v)) \searrow 0$ as $E(v) \searrow 0$. Indeed, then

\begin{equation}\label{9.3}
\| v \|_{X} \leq E(v)^{1/2} + c(E(v)) \| v \|_{X} + E(v)^{1/2} \| v \|_{X}^{2} + E(v)^{3/2},
\end{equation}

\noindent so for $E(v)$ sufficiently small, $\| v \|_{X} \leq E(v)^{1/2}$. The proof of $(\ref{9.2})$ will occupy the remainder of the paper.

\begin{lemma}\label{l10}
\begin{equation}\label{10.1}
\Big\| e^{2Z} \frac{R^{2}}{r^{2}} v^{3} \zeta(Rv) \Big\|_{L_{T}^{1} L_{R}^{2}(\mathbb{R} \times \mathbb{R}^{4})} \leq \| v \|_{X}^{2} E(v)^{1/2}.
\end{equation}
\end{lemma}

\noindent \emph{Proof:} This is straightforward. From \cite{diss, AGS}, for $E(v)$ sufficiently small, 

\begin{equation}
\Big|e^{2Z} - 1 \Big|, \hspace{5mm} \Big|\frac{R}{r} - 1\Big|, \hspace{5mm} \sup_{T, R} \Big|R v(T, R)\Big| \leq E(v),
\end{equation}

\noindent so

\begin{equation}\label{10.2}
\Big\| e^{2Z} \frac{R^{2}}{r^{2}} v^{3} \zeta(R v) \Big\|_{L_{T}^{1} L_{R}^{2}(\mathbb{R} \times \mathbb{R}^{4})} \leq \| v \|_{L_{T}^{2} L_{R}^{8}}^{2} \| v \|_{L_{T}^{\infty} L_{R}^{4}} \leq\| v \|_{X}^{2} E(v)^{1/2}.
\end{equation}

\noindent Next,

\begin{lemma}\label{l11}
\begin{equation}\label{11.1}
\Big\| \left(\partial_{\eta} r + \frac{1}{2}\right) \frac{v}{R^{2}} \Big\|_{L_{T}^{1} L_{R}^{2}(\mathbb{R} \times \mathbb{R}^{4})} \leq \| v \|_{X}^{2} E(v),
\end{equation}

\noindent and

\begin{equation}\label{11.2}
\Big\| \left(\partial_{\xi} r - \frac{1}{2}\right) \frac{v}{R^{2}}\Big \|_{L_{T}^{1} L_{R}^{2}(\mathbb{R} \times \mathbb{R}^{4})} \leq \| v \|_{X}^{2} E(v).
\end{equation}
\end{lemma}

\noindent \emph{Proof:} First take $(\ref{11.1})$. By the fundamental theorem of calculus,

\begin{equation}\label{11.3}
 \partial_{\eta} r \Big|_{R = 0} = -\frac{1}{2}, \hspace{5mm} \partial_{\xi} \left(\partial_{\eta} + \frac{1}{2}\right) r = \frac{e^{2Z}}{4} \frac{f^{2}(R v)}{r},
\end{equation}

\noindent $Z \sim 1$, $f(0) = 0$, $f'(0) = 1$, $| R v| \leq E(v)^{1/2}$, $\xi =  T + R$, $\eta = T - R$, then

\begin{equation}\label{11.4}
\left(\ptl_\xi r + \frac{1}{2} \right) \frac{v(T, R)}{R^{2}} \leq  \frac{1}{R^{2}} \left(\int_{0}^{R} v(T - R + s, s)^{2} \cdot s ds \right) v(T, R).
\end{equation}

\noindent Making a change of variables $s = \lambda R$, $0 \leq \lambda \leq 1$,

\begin{align}\label{11.5}
&\Bigg\| \left(\partial_{\eta} r + \frac{1}{2} \right) \frac{v(T, R)}{R^{2}} \Bigg\|_{L_{T}^{1} L_{R}^{2}(\mathbb{R} \times \mathbb{R}^{4})}  \notag\\
&\leq \Bigg\| \left(\int_{0}^{1} v(T + (\lambda - 1) R, \lambda R)^{2} \hspace{1mm} v \hspace{1mm} d\lambda \right) v(T, R) \Bigg\|_{L_{T}^{1} L_{R}^{2}(\mathbb{R} \times \mathbb{R}^{4})} \notag \\ 
&\leq \Big\| R^{1/4} v(T, R) \Big\|_{L_{T}^{2} L_{R}^{16}(\mathbb{R} \times \mathbb{R}^{4})} \cdot \\
& \quad \cdot \left(\int_{0}^{1} \Bigg\| \frac{1}{R^{1/4}} v(T+ (\lambda - 1) R, \lambda R)^{2} \Bigg\|_{L_{T}^{2} L_{R}^{16/7}(\mathbb{R} \times \mathbb{R}^{4})} \hspace{1mm} \lambda \hspace{1mm} d\lambda \right).
\end{align}

\noindent Doing a change of variables,

\begin{align}\label{11.6}
&\Bigg\| \frac{1}{R^{1/4}} v(T + (\lambda - 1) R, \lambda R) \Bigg\|_{L_{T, R}^{4}(\mathbb{R} \times \mathbb{R}^{4})}^{4} \notag\\ 
&= \int \int v(T + (\lambda - 1) R, \lambda R)^{4} R^{2} dR dT = \lambda^{-3} \int \int v(T, R)^{4} R^{2} dR dT.
\end{align}

\noindent By Hardy's inequality and the Sobolev embedding theorem,

\begin{align}\label{11.7}
&\lambda^{-3} \Big\| \frac{1}{R^{1/4}} v(T, R) \Big\|_{L_{T, R}^{4}(\mathbb{R} \times \mathbb{R}^{4})}^{4} \notag\\ 
\leq& \,\lambda^{-3} \Big\| \frac{1}{R} v(T, R) \Big\|_{L_{T}^{\infty} L_{R}^{2}(\mathbb{R} \times \mathbb{R}^{4})} \Big\|\, v (T, R)\, \Big\|_{L_{T}^{\infty} L_{R}^{4}(\mathbb{R} \times \mathbb{R}^{4})} \Big\| \,v (T, R)\, \Big\|_{L_{T}^{2} L_{R}^{8}(\mathbb{R} \times \mathbb{R}^{4})}^{2}  \notag\\
 \leq& \,E(v) \| v \|_{X}^{2}.
\end{align}

\noindent Meanwhile, by the Sobolev embedding theorem and interpolation,

\begin{align}\label{11.8}
&\Big\| v(T + (\lambda - 1) R, \lambda R) \Big\|_{L_{T}^{4} L_{R}^{16/3}(\mathbb{R} \times \mathbb{R}^{4})}^{4}  \\
\leq& \, \Big\| |\partial_{R}|^{1/4} v(T + (\lambda - 1) R, \lambda R) \Big\|_{L_{T, R}^{4}(\mathbb{R} \times \mathbb{R}^{4})}^{4} \\ &\lesssim \| \partial_{R} v(T + (\lambda - 1) R, \lambda R) \|_{L_{t}^{\infty} L_{R}^{2}} \| v(T, R) \|_{L_{T}^{3} L_{R}^{6}}^{3} \notag\\
\leq &\,\lambda^{-4} \left(\Big\| \partial_{T} v (T, R) \Big\|_{L_{T}^{\infty} L_{R}^{2}} + \Big\| \partial_{R} v (T, R)\Big \|_{L_{T}^{\infty} L_{R}^{2}}\right) \Big\|\, v (T, R) \,\Big\|_{L_{T}^{2} L_{R}^{8}}^{2} \Big\| \,v (T, R)\, \Big\|_{L_{T}^{\infty} L_{R}^{4}} \notag \\
\leq& \lambda^{-4} (T, R) \Big\|\, v (T, R) \,\Big\|_{X}^{2} E(v).
\end{align}

\noindent Plugging $(\ref{11.7})$ and $(\ref{11.6})$ into $(\ref{11.5})$,

\begin{equation}\label{11.10}
\Big\| \left(\partial_{\eta} r + \frac{1}{2} \right) \frac{v}{R^{2}} \Big\|_{L_{T}^{1} L_{R}^{2}(\mathbb{R} \times \mathbb{R}^{4})} \leq \| v\|_{X}^{2} E(v) \left(\int_{0}^{1} \lambda^{-3/4} d\lambda \right) \leq\| v \|_{X}^{2} E(v)^{1/2}.
\end{equation}

\noindent This takes care of $(\ref{11.1})$. The proof of $(\ref{11.2})$ is almost identical, this time integrating $\partial_{\eta} \partial_{\eta} r$ with respect to $\eta$ and utilizing $(\ref{11.3})$, and $\partial_{\eta} r\big|_{R = 0} = \frac{1}{2}$. $\Box$\vspace{5mm}

\noindent To compute

\begin{equation}\label{12.1}
(e^{2Z} - 1) \frac{v}{R^{2}},
\end{equation}

\noindent we will use the following `mass-aspect' function, 

\begin{equation}\label{12.1.1}
m \fdg= 1 + 4 e^{-2Z} \partial_{\xi} r \partial_{\eta} r.
\end{equation}

\noindent $\Big|e^{2Z} - 1\Big| \leq c(E(v))$ small implies that $\Big|1 - e^{-2Z}\Big|$ is small, so we can make the expansion

\begin{equation}\label{12.2}
e^{2Z} = \frac{1}{1 - (1 - e^{-2Z})} = \sum_{n = 0}^{\infty} (1 - e^{-2Z})^{n},
\end{equation}

\noindent and

\begin{equation}\label{12.3}
e^{2Z} - 1 = e^{2Z} (1 - e^{-2Z}) = \sum_{n = 0}^{\infty} (1 - e^{-2Z})^{n + 1}.
\end{equation}

\noindent The sums converge exponentially, so we will confine our computations to the leading order term in $(\ref{12.3})$, and estimate

\begin{equation}\label{12.4}
\left(1 - e^{-2Z}\right) \frac{v}{R^{2}}.
\end{equation}

\begin{equation}\label{12.5}
\left(1 - e^{-2Z}\right) \frac{v}{R^{2}} = \left(1 + 4 e^{-2Z} \partial_{\eta} r \partial_{\xi} r \right) \frac{v}{R^{2}} - 4 e^{-2Z}\left(\partial_{\eta} r \partial_{\xi} r + \frac{1}{4}\right) \frac{v}{R^{2}}.
\end{equation}

\noindent Now since

\begin{equation}\label{12.6}
\partial_{\eta} r \partial_{\xi} r + \frac{1}{4} = \partial_{\eta} r \left(\partial_{\xi} r - \frac{1}{2}\right) + \frac{1}{2} \left(\partial_{\eta} r + \frac{1}{2}\right),
\end{equation}

\noindent lemma $\ref{l11}$ implies

\begin{equation}\label{12.7}
\Bigg\| \left(\partial_{\eta} r \partial_{\eta} r + \frac{1}{4}\right) \frac{v}{R^{2}} \Bigg\|_{L_{T}^{1} L_{R}^{2}(\mathbb{R} \times \mathbb{R}^{4})} \leq \| v \|_{X}^{2} E(v)^{1/2},
\end{equation}

\noindent so it only remains to compute

\begin{equation}
\sum_{j} 2^{j/2} \Big\| m \frac{v}{R^{2}} \,\Big\|_{L_{T,x}^{2}(\mathbb{R} \times \{ x : 2^{j} \leq |x| \leq 2^{j + 1} \})}.
\end{equation}

\noindent By $(\ref{12.5})$ and $\Big|1 - e^{2Z}\Big|$, $\Big|\partial_{\eta} r + \frac{1}{2}\Big|$, $\Big|\partial_{\xi} r - \frac{1}{2}\Big| \leq c(E(v))$,

\begin{equation}\label{12.7.1}
\sup_{T, R} |m| \leq c(E(\omega)).
\end{equation}

\noindent Now make a spatial partition of unity. Suppose $\phi(x) \in C_{0}^{\infty}(\mathbb{R}^{4})$ is a radial, decreasing function, $\phi(x) = 1$ for $|x| \leq 1$, and $\phi(x)$ is supported on $|x| \leq 2$. Then let

\begin{equation}\label{12.9}
\chi(x) = \phi(\frac{x}{2}) - \phi(x).
\end{equation}

\noindent For any $x \neq 0$,

\begin{equation}\label{12.10}
\sum_{j \in \mathbf{Z}} \chi(2^{-j} x) = 1.
\end{equation}

\noindent Combining $(\ref{12.10})$ with the Littlewood - Paley decomposition,

\begin{equation}\label{12.11}
\aligned
P_{N} \left(m \frac{v}{R^{2}}\right) = \sum_{M} \sum_{j \in \mathbf{Z}} P_{N}\left(\chi(2^{-j} \rho) m \frac{P_{M} v}{R^{2}}\right) \\ = \sum_{M} P_{N} (\phi(N R) \frac{P_{M} v}{R^{2}}) + \sum_{M} \sum_{2^{j} \geq \frac{1}{N}} P_{N} \left(\chi(2^{-j} \rho) m \frac{P_{M} v}{R^{2}}\right).
\endaligned
\end{equation}

\noindent Since the Littlewood - Paley convolution kernel is rapidly decreasing,

\begin{align}\label{12.12}
&\sum_{2^{j} \geq \frac{1}{N}} \sum_{k} 2^{k/2} \Big\| P_{N}(\chi(2^{-j} R) \frac{P_{M} v}{R^{2}}) \Big\|_{L_{t,x}^{2}(\mathbb{R} \times \{ x : 2^{k} \leq |x| \leq 2^{k + 1} \})} \notag\\
\leq& \sum_{2^{j} \geq \frac{1}{N}} \sum_{k \geq j} 2^{k/2} 2^{-5|j - k|}\Big \| P_{N}(\chi(2^{-j} R)\frac{P_{M} v}{R^{2}}) \Big\|_{L_{T,x}^{2}(\mathbb{R} \times \mathbb{R}^{4})} \notag\\
&\quad\quad\quad\quad\quad + \sum_{2^{j} \geq \frac{1}{N}} \sum_{k \leq j} 2^{k/2} \Big\| P_{N}(\chi(2^{-j} R) \frac{P_{M} v}{R^{2}}) \Big\|_{L_{T,x}^{2}(\mathbb{R} \times \mathbb{R}^{4})} \notag\\
\leq& \sum_{2^{j} \geq \frac{1}{N}} 2^{j/2} \Big\| P_{N}(\chi(2^{-j} R) \frac{P_{M} v}{R^{2}}) \Big\|_{L_{T,x}^{2}(\mathbb{R} \times \mathbb{R}^{4})},
\end{align}

\noindent and

\begin{align}
&\sum_{k} 2^{k/2} \Big\| P_{N} (\phi(N R) \frac{P_{M} v}{R^{2}}) \Big\|_{L_{T,x}^{2}(\mathbb{R} \times \{ x : 2^{k} \leq |x| \leq 2^{k + 1} \})} \notag\\
\leq& \sum_{2^{k} \leq \frac{1}{N}} 2^{k/2} \Big\| P_{N} (\phi(N R) \frac{P_{M} v}{R^{2}}) \Big\|_{L_{T,x}^{2}(\mathbb{R} \times \mathbb{R}^{4})}  \notag \\
&+ \sum_{2^{k} \geq \frac{1}{N}} 2^{k/2} \left(\frac{2^{-k}}{N}\right)^{5} \Big\| P_{N}(\phi(N R) \frac{P_{M} v}{R^{2}}) \Big\|_{L_{T,x}^{2}(\mathbb{R} \times \mathbb{R}^{4})} \notag \\
\leq& N^{-1/2} \Big\| P_{N}(\phi(N R) \frac{P_{M} v}{R^{2}}) \Big\|_{L_{T,x}^{2}(\mathbb{R} \times \mathbb{R}^{4})}.
\end{align}

\noindent For each $N$ we will consider four cases, $M \geq N$ on the support of $\phi(N \rho)$, $M \geq N$ and $R \geq \frac{1}{N}$, $M \leq N$ on the support of $\phi(N R)$, and $M \leq N$ and $R \geq \frac{1}{N}$. We start with $M \geq N$ and $R \geq \frac{1}{N}$. By $(\ref{12.7.1})$,

\begin{align}
&\Big\| \sum_{M \geq N} \sum_{2^{j} \geq \frac{1}{N}} P_{N} (\chi(2^{-j} R) m \frac{P_{M} v}{R^{2}}) \Big\|_{Y} \notag\\ 
\leq &\sum_{M \geq N} \sum_{2^{j} \geq \frac{1}{N}} 2^{j/2} c(E(v)) \Big\| \chi(2^{-j} R) \frac{P_{M} v}{R^{2}} \Big\|_{L_{T,x}^{2}(\mathbb{R} \times \mathbb{R}^{4})} \label{12.13}\\
\leq & \sum_{M \geq N} \sum_{2^{j} \geq \frac{1}{N}} \frac{2^{-j}}{M} c(E(v)) \left(\sup_{\rho > 0} \rho^{-1/2} \cdot M \Big\| P_{M} v \Big \|_{L_{T,x}^{2}(\mathbb{R} \times \{ x : |x| \leq \rho \})} \right)
 \label{12.14} \\
 \leq & \sum_{M \geq N} \frac{N}{M} c(E(v)) \left(\sup_{\rho} \rho^{-1/2} \cdot M \Big\| P_{M} v \Big \|_{L_{T,x}^{2}(\mathbb{R} \times \{ |x| \leq \rho \})}\right) \label{12.15}
\end{align}



\noindent Then by Young's inequality and $(\ref{6.1})$, since we are summing $M$ and $N$ over the dyadic integers $M = 2^{k}$, $N = 2^{l}$ for $k, l \in \mathbb{Z}$,

\begin{align}\label{12.16}
&\sum_{N} \left(\sum_{M \geq N} \frac{N}{M} c(E(v)) \left(\sup_{\rho} \rho^{-1/2} \cdot M \Big\| P_{M} v \Big\|_{L_{T,x}^{2}(\mathbb{R} \times \{ |x| \leq \rho\})}\right)\right)^{2} \notag \\ 
&\leq c(E(v)^{2} \| v \|_{X}^{2}.
\end{align}

\noindent For $M \geq N$ on the support of $\phi(N R)$, the Sobolev embedding theorem and Holder's inequality imply that

\begin{align}
&\sum_{M \geq N} N^{-1/2} \Big \|  P_{N}\left(\phi(N R) m \frac{P_{M} v}{R^{2}}\right) \Big \|_{L_{T,x}^{2}(\mathbb{R} \times \mathbb{R}^{4})} \label{12.17}\\ 
\leq & N^{-1/2} N^{2} \sum_{M \geq N} \Big \| \phi(R N) m \frac{P_{M} v}{R^{2}} \Big \|_{L_{T}^{2} L_{x}^{1}(\mathbb{R} \times \mathbb{R}^{4})}
\label{12.18} \\
\leq & c(E(v)) N^{3/2} \sum_{M \geq N} \Big \| \frac{1}{R^{1/6}} P_{M} v \Big \|_{L_{T,x}^{2}(\mathbb{R} \times \{ x : |x| \leq \frac{1}{N} \})} \Bigg \| \frac{1}{R^{11/6}} \Bigg \|_{L_{T}^{\infty} L_{x}^{2}(\mathbb{R} \times \{ x : |x| \leq \frac{1}{N} \})} \\
\leq & c(E(v)) \sum_{M \geq N} \frac{N}{M} \left(\sup_{\rho> 0} \rho^{-1/2} \cdot M \Big \| P_{M} v \Big \|_{L_{T,x}^{2}(\mathbb{R} \times \{ x : |x| \leq \rho \})}\right). \label{12.19}
\end{align}

\noindent Again by Young's inequality,

\begin{align}\label{12.20}
&\sum_{N} \left(\sum_{M \geq N} \frac{N}{M} c(E(v)) \left(\sup_{\rho} \rho^{-1/2} \cdot M \Big \| P_{M} v \Big \|_{L_{T,x}^{2}(\mathbb{R} \times \{ |x| \leq \rho \})}\right)\right)^{2}  \notag \\
&\leq c(E(v))^{2} \| v \|_{X}^{2}.
\end{align}

\noindent Likewise, for $M \leq N$ and $R$ on the support of $\phi(N R)$, the Sobolev embedding theorem and H{\"o}lder's inequality imply

\begin{align}\label{12.21}
&\sum_{M \leq N} N^{-1/2} \Big \| P_{N}\left(\phi(R N) \frac{P_{M} v}{R^{2}}\right)\Big  \|_{L_{T,x}^{2}(\mathbb{R} \times \mathbb{R}^{4})} \\ 
\leq & \sum_{M \leq N} c(E(v)) \Big \| \phi(R N) \frac{P_{M} v}{R^{2}} \Big \|_{L_{T}^{2} L_{x}^{8/5}(\mathbb{R} \times \mathbb{R}^{4})} \\
\leq &\sum_{M \leq N} N^{-1/2} c(E(v)) \Big \| P_{M} v \Big \|_{L_{T}^{2} L_{x}^{\infty}(\mathbb{R} \times \mathbb{R}^{4})} \\ 
\leq & \sum_{M \leq N} \frac{M^{1/2}}{N^{1/2}} c(E(v)) \Big \| P_{M} v \Big \|_{L_{T}^{2} L_{x}^{8}(\mathbb{R} \times \mathbb{R}^{4})},
\end{align}

\noindent and

\begin{equation}\label{12.22}
\sum_{N} \left(\sum_{M \leq N} c(E(v)) \frac{M^{1/2}}{N^{1/2}} \Big \| P_{M} v\Big \|_{L_{T}^{2} L_{x}^{8}(\mathbb{R} \times \mathbb{R}^{4})}\right)^{2} \leq c(E(v))^{2} \| v \|_{X}^{2}.
\end{equation}

\noindent Finally suppose $R \geq \frac{1}{N}$ and $M \leq N$. By H\"older's inequality and the Sobolev embedding theorem,

\begin{align}
&\sum_{M \leq N} \sum_{\frac{1}{N} \leq 2^{j} \leq M^{-1/4} N^{-3/4}} 2^{j/2} \Big \| P_{N} \left(\chi(2^{-j} R) \frac{P_{M} v}{R^{2}}\right) \Big \|_{L_{T,x}^{2}(\mathbb{R} \times \mathbb{R}^{4})}
\label{12.24} \\
\leq & \sum_{M \leq N} \sum_{2^{j} \leq M^{-1/4} N^{-3/4}} \Big \| P_{M} v \Big \|_{L_{T}^{2} L_{x}^{\infty}(\mathbb{R} \times \mathbb{R}^{4})} 2^{j/2} c(E(v)) 
\\ 
\leq & c(E(v)) \sum_{M \leq N} \frac{M^{1/8}}{N^{1/8}} \Big \| P_{M} v \Big \|_{L_{T}^{2} L_{x}^{8}(\mathbb{R} \times \mathbb{R}^{4})},
\end{align}

\noindent and by Young's inequality,

\begin{equation}\label{12.25}
\sum_{N} \left(\sum_{M \leq N} \frac{M^{1/8}}{N^{1/8}} \Big \| P_{M} v \Big \|_{L_{T}^{2} L_{x}^{8}(\mathbb{R} \times \mathbb{R}^{4})} c(E(v))\right)^{2} \leq c(E(v))^{2} \| v\|_{X}^{2}.
\end{equation}

\noindent It only remains to compute

\begin{equation}\label{12.26}
\sum_{M \leq N} \sum_{2^{j} \geq M^{-1/4} N^{-3/4}} 2^{j/2} \Big \| P_{N} \left(m \frac{P_{M} v}{\rho^{2}}\right) \Big \|_{L_{T,x}^{2}(\mathbb{R} \times \mathbb{R}^{4})}.
\end{equation}

\noindent To compute this we will use Bernstein's inequality, which by the product rule will make use of a derivative of $m$. By Einstein's equations \eqref{eewmnull}, and $|\Omega^{2}| \leq 1$,

\begin{equation}\label{12.8}
\Big |\partial_{R} m \Big | \leq \frac{f(u)^{2}}{4r} + r (\partial u)^{2},
\end{equation}

\noindent where $(\partial u)^{2}$ is shorthand for $(\partial_{T} u)^{2} + |\grad_{x} u|^{2}$, $u= Rv$. By Bernstein's inequality, the product rule, and $\rho \sim R$,

\begin{align}
&2^{j/2} \Big \| P_{N} \left(m \chi(2^{-j} \rho) \frac{P_{M} v}{R^{2}}\right) \Big \|_{L_{T,x}^{2}(\mathbb{R} \times \mathbb{R}^{4})} \notag \\ \leq & \frac{2^{j/2}}{N} \Big \| \nabla P_{N} \left(m \chi(2^{-j} \rho) \frac{P_{M} v}{R^{2}}\right) \Big \|_{L_{T,x}^{2}(\mathbb{R} \times \mathbb{R}^{4})} \label{12.27}\\
\leq &\frac{2^{j/2}}{N} \Big \| P_{N} \left( m \chi(2^{-j} \rho ) \frac{\nabla P_{M} v}{R^{2}}\right) \Big \|_{L_{T,x}^{2}(\mathbb{R} \times \mathbb{R}^{4})} \label{12.29}\\
 +& \frac{2^{j/2}}{N} \Big \| P_{N} \left(m \chi(2^{-j} \rho) \frac{P_{M} v}{R^{3}}\right) \Big \|_{L_{T,x}^{2}(\mathbb{R} \times \mathbb{R}^{4})} \label{12.30}\\
 +& \frac{2^{-j/2}}{N} \Big \| P_{N} \left(m \chi'(2^{-j} \rho) \frac{ P_{M} v}{R^{2}}\right) \Big \|_{L_{T,x}^{2}(\mathbb{R} \times \mathbb{R}^{4})} \label{12.31}\\
  +& \frac{2^{j/2}}{N} \Big \| P_{N} \left((\partial_{R} m) \chi(2^{-j} \rho) \frac{P_{M} v}{R^{2}}\right) \Big \|_{L_{T,x}^{2}(\mathbb{R} \times \mathbb{R}^{4})}. \label{12.28}
\end{align}

\noindent First take $(\ref{12.29})$.

\begin{equation}\label{12.32}
\aligned
(\ref{12.29}) \leq c(E(v)) 2^{-2j} \cdot \frac{2^{j/2}}{N} \Big \| \nabla P_{M} v \Big \|_{L_{T,x}^{2}(\mathbb{R} \times \{ 2^{j} \leq |x| \leq 2^{j + 1} \})} \\ \leq c(E(v)) \frac{2^{-j}}{N} \left(\sup_{\rho} \rho^{-1/2} \Big \| \nabla P_{M} v \Big \|_{L_{T,x}^{2}(\mathbb{R} \times \{ x : |x| \leq \rho \})}\right).
\endaligned
\end{equation}

\begin{equation}\label{12.33}
\aligned
c(E(v)) \sum_{2^{j} \geq M^{-1/4} N^{-3/4}} \frac{2^{-j}}{N} \left(\sup_{\rho} \rho^{-1/2} \Big \| \nabla P_{M} v \Big \|_{L_{T,x}^{2}(\mathbb{R} \times \{ x : |x| \leq \rho) \}}\right) \\ \leq c(E(v)) \frac{M^{1/4}}{N^{1/4}} \left(\sup_{\rho} \rho^{-1/2} \Big \| \nabla P_{M} v \Big \|_{L_{T,x}^{2}(\mathbb{R} \times \{ x : |x| \leq \rho \})}\right),
\endaligned
\end{equation}

\noindent and by Young's inequality,

\begin{equation}\label{12.34}
\aligned
\sum_{N} (c(E(v)) \sum_{M \leq N} \frac{M^{1/4}}{N^{1/4}} \left(\sup_{\rho} \rho^{-1/2} \Big \| \nabla P_{M} v \Big \|_{L_{T,x}^{2}(|x| \leq \rho)})\right)^{2} \leq c(E(v))^{2} \| v\|_{X}^{2}.
\endaligned
\end{equation}

\noindent Next take $(\ref{12.30})$. This time, by H\"older's inequality,

\begin{equation}\label{12.35}
\frac{2^{j/2}}{N} \Bigg \| \frac{\chi(2^{-j} \rho)}{R^{3}} m P_{M} v \Big \|_{L_{T,x}^{2}(\mathbb{R} \times \mathbb{R}^{4})}
\end{equation}

\begin{equation}\label{12.36}
\leq \frac{2^{-j/2}}{N} c(E(v)) \Big \| P_{M} v \Big \|_{L_{T}^{2} L_{x}^{\infty}(\mathbb{R} \times \mathbb{R}^{4})} \leq \frac{2^{-j/2} M^{1/2}}{N} c(E(v)) \Big \| P_{M} v \Big \|_{L_{T}^{2} L_{x}^{8}(\mathbb{R} \times \mathbb{R}^{4})}.
\end{equation}

\noindent Again by Young's inequality,

\begin{align}\label{12.37}
&\sum_{N} \left(\sum_{M \leq N} \sum_{2^{j} \geq C M^{-1/4} N^{-3/4}} \frac{2^{-j/2} M^{1/2}}{N} c(E(v)) \Big \| P_{M} v \Big \|_{L_{T}^{2} L_{x}^{8}(\mathbb{R} \times \mathbb{R}^{4})}\right)^{2}
\notag \\ 
&\leq c(E(v))^{2} \| v \|_{X}^{2}.
\end{align}

\noindent The estimate of $(\ref{12.31})$ is virtually identical to the estimate of $(\ref{12.30})$.\vspace{5mm}

\noindent Finally, we turn our attention to $(\ref{12.28})$. $f(0) = 0$, $f'(0) = 1$, $R v$ uniformly bounded implies that the first term in $(\ref{12.8})$,

\begin{equation}\label{12.38}
\frac{f(u)^{2}}{r} \leq \frac{c(E(v))}{R},
\end{equation}

\noindent and then

\begin{equation}
\frac{2^{j/2}}{N} \Bigg \| \frac{f(u)^{2}}{r} \chi(2^{-j} \rho) \frac{P_{M} v}{R^{2}} \Bigg \|_{L_{T,x}^{2}(\mathbb{R} \times \mathbb{R}^{4})}
\end{equation}

\noindent can be computed in exactly the same manner as $(\ref{12.30})$ or $(\ref{12.31})$.\vspace{5mm}






\noindent Now we compute

\begin{equation}
\frac{2^{j/2}}{N} \Bigg \| P_{N} \left(\chi(2^{-j} \rho) \frac{P_{M} v}{R^{2}} r (\partial u)^{2}\right) \Bigg \|_{L_{T,x}^{2}(\mathbb{R} \times \mathbb{R}^{4})}.
\end{equation}

\noindent By the radial Sobolev embedding theorem and lemma $\ref{l7}$,

\begin{align}\label{12.39}
&\frac{2^{j/2}}{N} \Big \| (1 - \phi(2^{-j + 10} \rho)) P_{N} \left(\chi(2^{-j} \rho) r^{3} (\partial v)^{2} \frac{P_{M} v}{R^{2}} \right) \Big \|_{L_{T,x}^{2}(\mathbb{R} \times \mathbb{R}^{4})} \\ 
\leq &\, \frac{2^{j/2}}{N} \Big \| (1 - \phi(2^{-j + 10} \rho)) P_{N}\left(\chi(2^{-j} \rho) (\partial v)^{2} R P_{M} v \right) \Big \|_{L_{T,x}^{2}(\mathbb{R} \times \mathbb{R}^{4})} \\
\leq  & \,\frac{2^{-j}}{N^{1/2}} \Big \| \chi(2^{-j} \rho) (\partial v)^{2} (R P_{M} v) \Big \|_{L_{T}^{2} L_{x}^{1}(\mathbb{R} \times \mathbb{R}^{4})} \\ 
\leq & \, \frac{2^{-j/2}}{N^{1/2}} E(v)^{1/2} \left(\sup_{\rho > 0} \rho^{-1/2} \| \partial v \|_{L_{T,x}^{2}(\mathbb{R} \times \{ x : |x| \leq \rho \})}\right) \Big \| R P_{M} v\Big \|_{L_{T,x}^{\infty}(\mathbb{R} \times \mathbb{R}^{4})} \\
\leq & \, \frac{2^{-j/2}}{N^{1/2}} E(v)^{1/2} \| \,v\,  \|_{X} \,\cdot\, \Big \| R P_{M} v \Big \|_{L_{T,x}^{\infty}(\mathbb{R} \times \mathbb{R}^{4})}.
\end{align}

\noindent Therefore by Young's inequality,

\begin{align}\label{12.40}
&\sum_{N} \| v \|_{X}^{2} \left(\sum_{M \leq N} \sum_{2^{j} \geq M^{-1/4} N^{-3/4}} E(v)^{1/2} \frac{2^{-j/2}}{N^{1/2}} \Big \| R P_{M} v \Big \|_{L_{T,x}^{\infty}(\mathbb{R} \times \mathbb{R}^{4})}\right)^{2} \notag \\ 
&\leq E(v) \| v\|_{X}^{4}.
\end{align}



\noindent Now if $x \in \text{supp} (\phi(2^{-j + 10} \rho))$ and $y \in \text{supp} (\chi(2^{-j} \rho))$, $|x - y| \sim 2^{j}$. Therefore, since the Littlewood - Paley projection kernel is rapidly decreasing,

\begin{align}
&\frac{2^{j/2}}{N} \Bigg \| \phi(2^{-j + 10} \rho) P_{N} \left(\chi(2^{-j} \rho) r^{3} (\partial v)^{2} \frac{P_{M} v}{R^{2}}\right) \Bigg\|_{L_{T,x}^{2}(\mathbb{R} \times \mathbb{R}^{4})} \label{12.42} \\ 
\leq &\frac{2^{-5j/2}}{N^{4}} \Big\| \chi(2^{-j} \rho) (\partial v)^{2} R (P_{M} v) \Big\|_{L_{T,x}^{2}(\mathbb{R} \times \mathbb{R}^{4})} \label{12.43}\\
\leq &\frac{2^{-2j}}{N^{2}} \Big\| R P_{M} v \Big \|_{L_{T,x}^{\infty}(\mathbb{R} \times \mathbb{R}^{4})} \left(\sup_{\rho > 0} \rho^{-1/2} \Big\| \, \partial v \, \Big \|_{L_{T,x}^{2}(\mathbb{R} \times \{ x : |x| \leq  \rho \})} \right) E(v)^{1/2}.
\end{align}

\noindent Once more, by Young's inequality and Hardy's inequality,

\begin{align}\label{12.44}
&\sum_{N} \left( \sum_{M \leq N} \left(\sum_{2^{j} \geq N^{-3/4} M^{-1/4}} \frac{2^{-2j}}{N^{2}} \| v \|_{X} E(v)^{1/2} \Big\| R P_{M} v\Big\|_{L_{T,x}^{\infty}(\mathbb{R} \times \mathbb{R}^{4})} \right) \right)^{2} \notag\\ 
&\leq E(v) \| v \|_{X}^{4}.
\end{align}

\noindent Combining $(\ref{12.16})$, $(\ref{12.20})$, $(\ref{12.22})$, $(\ref{12.25})$, $(\ref{12.34})$, $(\ref{12.37})$, $(\ref{12.40})$, and $(\ref{12.44})$ proves $(\ref{9.2})$, which in turn completes the proof of Theorem $\ref{t9}$. $\Box$

\section{Scattering for Problem II}
In this section we consider the radial wave equation

\begin{equation} \label{Wave2-2}
\left. \begin{array}{rcl}
\leftexp{4+1}{\square}\, \tilde{v} &=& \tilde{F}(\tilde{v})\,\,\,\,\,\,\,\,\,\,\,\,\,\,\, \,\,\,\,\,\,\,\,\,\,\,\,\,\,\,\,\,\,\textnormal{on}\,\, \mathbb{R}^{4+1}\\
 \tilde{v}_0 = \tilde{v} (0, x) & \textnormal{and}& \tilde{v}_1 = \ptl_T \tilde{v} (0, x) \,\,\,\,\,\,\,\, \textnormal{on}\,\, \mathbb{R}^4\\\end{array} 
\right\}
\end{equation}
where 
\begin{align}
\tilde{F}(\tilde{v}) =& \left( \frac{1}{r} \ptl_\eta r + \frac{1}{2R} \right)\ptl_\xi \tilde{v}  + \left( \frac{1}{r} \ptl_\xi r - \frac{1}{2R} \right)\ptl_\eta \tilde{v} \notag\\
&\quad + \left( \left(\frac{r}{R} \ptl_\eta r + \halb \right) - \left(\frac{r}{R} \ptl_\xi r - \halb \right) \right) \frac{\tilde{v}}{r^2} \notag\\
&\quad + \frac{R^2}{r^2} \tilde{v}^3 \zeta(R \tilde{v})
\end{align}
and $\tilde{v}$ is coupled to Einstein's equations \eqref{eewm} with $u = R \tilde{v}.$  Define,

\begin{align}\label{def-energy-II}
\tilde{E}(\tilde{v}) \fdg = \Vert \tilde{v}_0 \Vert_{H^1(\mathbb{R}^4)} + \Vert \tilde{v}_1 \Vert_{L^2(\mathbb{R}^4)} + \halb \Vert \tilde{v}_0\Vert_{L^4(\mathbb{R}^4)},
\end{align}

Suppose 
\[ \Big \vert \frac{R}{r} -1 \Bigg\vert \leq \tilde{E}(\tilde{v}) \quad
\textnormal{and} \quad \vert R \tilde{v} \vert \leq \tilde{E}(\tilde{v}). \]
Recall that the equation \ref{Wave2-2} is a partially linearized equation of the original wave maps equation obtained by the linearization of the wave equation \ref{eq:waveZ} for $Z$ (which implies 
$Z \equiv 0$).

Firstly we prove the following nonlinear Morawetz estimate for small energy.

\begin{lemma}
For any global solution $\tilde{v}$ of \eqref{Wave2-2} such that 
\[ \Big \vert \frac{R}{r} -1 \Bigg\vert \leq \tilde{E}(\tilde{v}) \quad
\textnormal{and} \quad \vert R \tilde{v} \vert \leq \tilde{E}(\tilde{v}), \]
\begin{align}
\int_{\mathbb{R}^{4+1}}  \frac{\tilde{v}^2}{\vert x \vert^3} \bar{\mu}_{\check{g}} \leq \tilde{E} (\tilde{v})
\end{align}

for $\tilde{E} (\tilde{v}) <\eps^2,$ $\eps$ sufficiently small.
\end{lemma}
\begin{proof}
\noindent We shall use the estimates $\vert \frac{R}{r} -1 \vert, \vert R \tilde{v} \vert \leq \tilde{E}(\tilde{v})$ throughout.  
Define the Morawetz quantity

\begin{equation}\label{1.3}
M(T) \fdg = \int \tilde{v}_{R} \tilde{v}_{T} R^{3} dR + \frac{3}{2} \int \tilde{v}_{T} \tilde{v} R^{2} dR.
\end{equation}

\noindent Taking the time derivative,

\begin{align}\label{1.4}
\frac{d}{dt} (M(t)) =& \int \partial_{R} (\tilde{v}_{T}) \tilde{v}_{T} R^{3} dR + \frac{3}{2} \int \tilde{v}_{T}^{2} R^{2} dR + \int \tilde{v}_{R} \tilde{v}_{TT} R^{3} dR \notag\\
&\quad +\frac{3}{2} \int \tilde{v}_{TT} \tilde{v} R^{2} dR.
\end{align}

\noindent Integrating by parts,

\begin{equation}\label{1.5}
\int \partial_{R} (\tilde{v}_{T}) \tilde{v}_{T} R^{3} dR + \frac{3}{2} \int \tilde{v}_{T}^{2} R^{2} dR = 0.
\end{equation}

\noindent Now using \eqref{Wave2-2} to split $\tilde{v}_{TT}$,

\begin{align}
&\int \tilde{v}_{R} (\tilde{v}_{RR} + \frac{3}{R} \tilde{v}_{R}) R^{3} dR + \frac{3}{2} \int (\tilde{v}_{RR} + \frac{3}{R} \tilde{v}_{R}) \tilde{v} R^{2} dR \label{1.6}\\
&= -\frac{3}{2} \int \tilde{v}_{R}^{2} R^{2} dR + 3 \int \tilde{v}_{R}^{2} R^{2} dR - \frac{3}{2} \int \tilde{v}_{R}^{2} R^{2} dR \notag\\
 &\quad\quad - 3 \int \tilde{v}_{R} u R dR + \frac{9}{2} \int \tilde{v}_{R} \tilde{v} R dR \label{1.7}\\
&= -\frac{3}{4} \int \tilde{v}^{2} dR.
\end{align}

\noindent Therefore,

\begin{equation}\label{1.8}
\int \int \tilde{v}^{2} dR dT = \frac{4}{3} (M(T) - M(0)) + \int \int F(\tilde{v}) \tilde{v}_{R} R^{3} dR dT + \int \int F(\tilde{v}) \tilde{v} R^{2} dR dT.
\end{equation}

\noindent By Hardy's inequality and conservation of energy,

\begin{equation}\label{1.9}
|M(T) - M(0)| \leq \tilde{E}(\tilde{v}).
\end{equation}

\noindent First consider

\begin{equation}\label{1.10}
\int \int \tilde{F}(\tilde{v}) \tilde{v}_{R} R^{3} dR dT.
\end{equation}

\noindent Making a change of variables

\begin{equation}\label{1.11}
\aligned
\int \int \frac{1}{R} (\partial_{\eta} r - \frac{1}{2}) (\partial_{\xi} \tilde{v})(\partial_{R} \tilde{v}) R^{3} dR dT =& \int \int \frac{1}{R} (\partial_{\eta} r - \frac{1}{2}) (\partial_{\xi} \tilde{v})(\partial_{R} \tilde{v}) R^{3} d\eta d\xi \\
\leq& \int (\sup_{R > 0} \frac{1}{R} \int_{0}^{R} f^{2}(\tilde{v}) s ds)(\int (\partial \tilde{v})^{2} d\xi) d\eta \\ \leq& \tilde{E}(\tilde{v}) \int \int \tilde{v}^{2} dR dT.
\endaligned
\end{equation}

\noindent Similarly,

\begin{equation}\label{1.12}
\int \int \frac{1}{R} (\partial_{\xi} r + \frac{1}{2})(\partial_{\eta} \tilde{v})(\partial_{R} \tilde{v}) R^{3} dR dT \leq \tilde{E}(\tilde{v}) \int \int \tilde{v}^{2} dR dT.
\end{equation}

\noindent Next,

\begin{align}\label{1.13}
\int \int (\partial_{\eta} r + \frac{1}{2}) \frac{\tilde{v}}{R^{2}} \tilde{v}_{R} R^{3} dR dT \leq& \int \int \frac{1}{R} (\partial_{\eta} r + \frac{1}{2}) \tilde{v}_{R}^{2} R^{3} dR dT \notag \\
&\quad + \int \int (\partial_{\eta} r + \frac{1}{2}) \tilde{v}^{2} dR dT \notag\\
\leq& \tilde{E}(\tilde{v}) \int \int \tilde{v}^{2} dR dT + \epsilon \int \int \tilde{v}^{2} dR dT.
\end{align}

\noindent Likewise,

\begin{align}\label{1.14}
\int \int (\partial_{\xi} r - \frac{1}{2}) \frac{\tilde{v}}{R^{2}} \tilde{v}_{R} R^{3} dR dT \leq \tilde{E}(\tilde{v}) \int \int \tilde{v}^{2} dR dT + \epsilon \int \int \tilde{v}^{2} dR dT.
\end{align}

\noindent Expanding $\zeta(R\tilde{v}) = c_{1} + c_{2} (R\tilde{v}) +  \cdots$, then integrating by parts

\begin{equation}\label{1.16}
\aligned
c_{1} \int \int \frac{R^{2}}{r^{2}} \tilde{v}^{3} \tilde{v}_{R} R^{3} dR dT = \frac{c_{1}}{4} \int \int \frac{R^{5}}{r^{2}} \partial_{R}(\tilde{v}^{4}) dR dT \\
= -\frac{5c_{1}}{4} \int \int \frac{R^{4}}{r^{2}} \tilde{v}^{4} dR dT - \frac{c_{1}}{2} \int \int \frac{R^{5}}{r^{3}} (\partial_{R} r) \tilde{v}^{4} dR dT.
\endaligned
\end{equation}

\noindent Then by the radial Sobolev embedding theorem, $R \tilde{v} \leq \epsilon$. Therefore,

\begin{equation}
(\ref{1.16}) \lesssim \epsilon^{2} \int \int \tilde{v}^{2} dR dT.
\end{equation}

\noindent Now we turn to

\begin{equation}\label{1.17}
\int \int \tilde{F}(\tilde{v}) \tilde{v} R^{2} dR dT.
\end{equation}

\noindent First,

\begin{equation}\label{1.18}
\aligned
&\int \int \frac{1}{R} (\partial_{\eta} r - \frac{1}{2}) (\partial_{\xi} \tilde{v}) \tilde{v} R^{2} dR dT \\
&\leq \int \int \frac{1}{R} (\partial_{\eta} r - \frac{1}{2}) (\partial_{\xi} \tilde{v})^{2} R^{3} dR dT + \int \int (\partial_{\eta} r - \frac{1}{2}) \tilde{v}^{2} dR dT \\
&\leq \tilde{E}(\tilde{v}) \int \int \tilde{v}^{2} dR dT + \epsilon \int \int \tilde{v}^{2} dR dT.
\endaligned
\end{equation}

\noindent Similarly,

\begin{equation}\label{1.19}
\aligned
\int \int \frac{1}{R} (\partial_{\xi} r + \frac{1}{2}) (\partial_{\eta} \tilde{v}) \tilde{v} R^{2} dR dT
\leq \tilde{E}(\tilde{v}) \int \int \tilde{v}^{2} dR dT + \epsilon \int \int \tilde{v}^{2} dR dT.
\endaligned
\end{equation}
\noindent Therefore, by the fundamental theorem of calculus,

\begin{equation}\label{1.21}
\int \int \tilde{v}^{2} dR dT \leq \tilde{E}(\tilde{v}) + \epsilon \int \int \tilde{v}^{2} dR dT + \tilde{E}(\tilde{v}) \int \int \tilde{v}^{2} dR dT.
\end{equation}

\noindent Therefore, for $\tilde{E}(\tilde{v})$ sufficiently small,

\begin{equation}\label{1.22}
\int \int \tilde{v}^{2} dR dT \leq E(\tilde{v}).
\end{equation}
\end{proof}
\noindent Now then, it is necessary to prove scattering.

\begin{theorem}\label{t1.1}
The globally regular solution to $\eqref{Wave2-2}$ scatters forward and backward in time.
\end{theorem}

\noindent \emph{Proof:} To prove scattering it suffices to show that the solution to

\begin{equation}\label{1.23}
\left. \begin{array}{rcl}
\leftexp{4+1}{\square}\, \bar{v} &=&F(\tilde{v})\,\,\,\,\,\,\,\,\,\,\,\,\,\,\, \,\,\,\,\,\,\,\,\,\,\,\,\,\,\,\,\,\,\,\,\,\,\,\,\,\,\textnormal{on}\,\, \mathbb{R}^{4+1}\\
 \bar{v} (T_0, x)=0 & \textnormal{and}& \ptl_T \bar{v} (T_0, x)=0 \,\,\,\,\,\,\,\,\quad \,\textnormal{on}\,\, \mathbb{R}^4\\\end{array} 
\right\}
\end{equation}

where $F(\tilde{v})$ is
\begin{align}
F(\tilde{v}) =&  \left(\frac{1}{r}\partial_{\eta} r + \frac{1}{2R}\right) (\partial_{\xi} \tilde{v}) + \left(\frac{1}{R}\partial_{\xi} r - \frac{1}{2R}\right) (\partial_{\eta} \tilde{v}) \\
&\quad + \left(\frac{r}{R} \partial_{\eta} r + \frac{1}{2} - \frac{r}{R}\partial_{\xi} r - \frac{1}{2}\right) \frac{\tilde{v}}{r^{2}} + \frac{R^{2}}{r^{2}} \tilde{v}^{3} \zeta(R\tilde{v}), 
\end{align}

\noindent has a solution with \[ \sup_{T \geq T_{0}} [\| \bar{v}(T) \|_{\dot{H}^{1}} + \| \bar{v}_{T}(T) \|_{L^{2}}] \rightarrow 0\] as $T_{0} \rightarrow \infty$. Then

\begin{equation}\label{1.24}
\tilde{v}(T) = \bar{v}(T) + S(T - T_{0}) (\tilde{v}(T_{0}), \tilde{v}_{T}(T_{0})) = \bar{v}(T) + w(T),
\end{equation}

\noindent where $S(t)(\tilde{v}_0, \tilde{v}_1)$ is the solution to the wave equation $\Box w = 0$ with initial data $(\tilde{v}_0, \tilde{v}_1)$. In particular, this implies $\bar{E}(\bar{v}) \leq \tilde{E}(\tilde{v})$, where

\begin{equation}
\bar{E}(\bar{v}) = \| \bar{v}_{T} \|_{L^{2}}^{2} + \| \nabla \bar{v}\|_{L^{2}}^{2} + \halb \| \bar{v} \|_{L^{4}}^{4}.
\end{equation}

 Now compute

\begin{equation}\label{1.25}
\frac{d}{dT} [\frac{1}{2} \langle \bar{v}_{T}, \bar{v}_{T} \rangle + \frac{1}{2} \langle \nabla \bar{v}, \nabla \bar{v} \rangle ] = -\langle \bar{v}_{T}, \tilde{F}(\tilde{v}) \rangle.
\end{equation}

where $\ip{x}{y} = \int_{R^{4+1}} x \cdot y dRdT.$

\noindent Then as in $(\ref{1.11})$ and $(\ref{1.12})$, using $\vert \frac{R}{r} -1 \vert \leq \tilde{E}(\tilde{v})$, by the dominated convergence theorem,

\begin{equation}\label{1.26}
\lim_{T_{0} \rightarrow \infty} \int_{T_{0}}^{\infty} \int \frac{1}{R} (\partial_{\eta} - \frac{1}{2}) (\partial_{\xi} \tilde{v}) \cdot \bar{v}_{T} dR dT = 0,
\end{equation}

\noindent and

\begin{equation}\label{1.27}
\lim_{T_{0} \rightarrow \infty} \int_{T_{0}}^{\infty} \int \frac{1}{R} (\partial_{\xi} + \frac{1}{2}) (\partial_{\eta} \tilde{v}) \cdot \bar{v}_{T} dR dT = 0.
\end{equation}

\noindent As in $(\ref{1.13})$ and $(\ref{1.14})$,

\begin{equation}\label{1.28}
\lim_{T_{0} \rightarrow \infty} \int_{T_{0}}^{\infty} \int (\partial_{\eta} r + \frac{1}{2}) \frac{\tilde{v}}{R^{2}} \bar{v}_{T} dR dT = 0,
\end{equation}

\begin{equation}\label{1.29}
\lim_{T_{0} \rightarrow \infty} \int_{T_{0}}^{\infty} \int (\partial_{\xi} r - \frac{1}{2}) \frac{\tilde{v}}{R^{2}} \bar{v}_{T} dR dT = 0.
\end{equation}

\noindent To compute

\begin{equation}\label{1.30}
\int \int \frac{R^{2}}{r^{2}} \tilde{v}^{3} \zeta(R\tilde{v}) R^{3} \bar{v}_{T} dR dT,
\end{equation}

\noindent again expand out $\zeta(R\tilde{v})$. Then

\begin{equation}\label{1.31}
c_{1} \int \frac{R^{2}}{r^{2}} \bar{v}^{3} \bar{v}_{T} R^{3} dR = \frac{d}{dT} \left(\frac{c_{1}}{4} \int \frac{R^{2}}{r^{2}} \bar{v}^{4} R^{3} dR \right) + \frac{c_{1}}{2} \int \frac{R^{2}}{r^{3}} \bar{v}^{4} R^{3} (\partial_{T} r) dR.
\end{equation}

\noindent Now by the radial Sobolev embedding theorem, combined with the Morawetz estimates,

\begin{equation}\label{1.32}
\int_{T \geq T_{0}} \int \frac{R^{2}}{r^{3}} \bar{v}^{4} R^{3} (\partial_{T} r) dR dT \rightarrow 0
\end{equation}

\noindent as $T_{0} \rightarrow \infty$. Next, using the radial Strichartz estimate

\begin{equation}
\Big\Vert |x|^{1/2} S(t)(\tilde{v}_0, \tilde{v}_1) \Big\Vert_{L_{t}^{2} L_{x}^{\infty}} \leq \Vert  \tilde{v}_0 \Vert_{\dot{H}^{1}} + \Vert \tilde{v}_1 \Vert_{L^{2}},
\end{equation}

\noindent so

\begin{align}\label{1.33}
\int_{T \geq T_{0}} \int \frac{R^{2}}{r^{2}} \tilde{v}^{2} w \bar{v}_{T} R^{3} dR dT \leq& \left(\int_{T \geq T_{0}} \int \tilde{v}^{2} dR dT \right)^{1/2} \notag     \\
 & \quad \cdot\| R^{1/2} w \|_{L_{T}^{2} L_{x}^{\infty}} \| \bar{v}_{T} \|_{L_{T}^{\infty} L_{x}^{2}} \| R \tilde{v} \|_{L_{T, x}^{\infty}}.
\end{align}

\noindent If 

\begin{equation}\label{1.34}
E(\bar{v}(T)) = \| \bar{v}_{T} \|_{L^{2}}^{2} + \| \nabla \bar{v}(T) \|_{L^{2}}^{2} + \halb \| \bar{v}(T) \|_{L^{4}}^{4},
\end{equation}

\noindent then the Sobolev embedding theorem implies that for small energy,

\begin{equation}\label{1.35}
(\ref{1.33}) \leq (\sup_{T \geq T_{0}} E(\bar{v}(T)))^{1/2} \tilde{E}(\tilde{v}) \left(\int_{T \geq T_{0}} \int \tilde{v}^{2} dR dT \right)^{1/2}.
\end{equation}

\noindent Also,

\begin{equation}\label{1.36}
\aligned
\int_{T \geq T_{0}} \int \frac{R^{2}}{r^{2}} \tilde{v} \cdot w \cdot \bar{v} \bar{v}_{T} R^{3} dR dT \leq & \left(\int_{T \geq T_{0}} \int \tilde{v}^{2} dR dT \right)^{1/2} \\
& \cdot \| R^{1/2} w \|_{L_{T}^{2} L_{x}^{\infty}} \| \bar{v}_{T} \|_{L_{T}^{\infty} L_{x}^{2}} \| R \bar{v} \|_{L_{T, x}^{\infty}} \\ 
\leq & \tilde{E}(\tilde{v}) \left(\sup_{T \geq T_{0}} E(\bar{v}(T)) \right).
\endaligned
\end{equation}

\noindent Finally,

\begin{equation}\label{1.37}
\aligned
\int_{T \geq T_{0}} \int \frac{R^{2}}{r^{2}} w \bar{v}^{2} \bar{v}_{T} R^{3} dR dT \leq & \Big\Vert |x|^{1/2} w \Big\Vert_{L_{T}^{2} L_{x}^{\infty}} \| v_{T} \|_{L_{T}^{\infty} L_{x}^{2}} \\
& \quad \cdot \left(\int \tilde{v}^{2} + w^{2} dR dT \right)^{1/2} \| |x| \bar{v} \|_{L_{T, x}^{\infty}} \\ 
\leq & \tilde{E}(\tilde{v}) \left(\sup_{T \geq T_{0}} E(\bar{v})(T)\right).
\endaligned
\end{equation}

\noindent Therefore we have proved

\begin{equation}\label{1.38}
\sup_{T \geq T_{0}} E(\bar{v}(T)) \leq \epsilon \left(\sup_{T \geq T_{0}} E(\bar{v}(T)) \right) + \left(\int_{T \geq T_{0}} \int \tilde{v}^{2} dR dT \right).
\end{equation}

\noindent Since

\begin{equation}\label{1.39}
\int_{T \geq T_{0}} \int \tilde{v}^{2} dR dT \rightarrow 0
\end{equation}

\noindent as $T_{0} \rightarrow \infty$, we have scattering. $\Box$
\subsection*{Acknowledgements}
The authors are grateful to Jalal Shatah for answering their questions related to his work.  The authors are also thankful to the referee for the thoughtful comments and for suggesting an error in a previous work. 

\bibliography{final.bib}
\bibliographystyle{plain}

\end{document}